\RequirePackage{fix-cm}

\documentclass[smallcondensed,final]{svjour3} 
\usepackage[demo]{graphicx}
\usepackage{booktabs} 
\usepackage{tikz}
\usepackage{multirow}
\usepackage{pgfplots}
\usepackage{mathrsfs}
\usepackage{amsfonts, amssymb, amsmath}
\usepackage{algorithm}
\usepackage{algorithmic}
\usepackage[graphicx]{realboxes}
\usepackage{caption}
\usepackage{subcaption}
\usepackage{color}

\smartqed  
\allowdisplaybreaks
\begin{document}

\title{Robust solutions for stochastic and distributionally robust chance-constrained binary knapsack problems \thanks{This research was supported by the National Research Foundation of Korea (NRF) Grant funded by the Korea government (MSIT) (No. 2019R1F1A1061361).}
}
\titlerunning{Robust solutions for chance-constrained binary knapsack problem}  

\author{Jaehyeon Ryu \and Sungsoo Park}

\institute{Sungsoo Park \at
              Department of Industrial and Systems Engineering, KAIST, 291 Daehak-ro, Yuseong-gu, Daejeon 34141, Republic of Korea \\
              \email{sspark@kaist.ac.kr}           
           \and
           Jaehyeon Ryu \at Department of Industrial and Systems Engineering, KAIST, 291 Daehak-ro, Yuseong-gu, Daejeon 34141, Republic of Korea \\
              \email{jhryu034@gmail.com}              
}
\date{Received: date / Accepted: date}

\maketitle

\begin{abstract}
We consider chance-constrained binary knapsack problems, where the weights of items are independent random variables with the means and standard deviations known. The chance constraint can be reformulated as a second-order cone constraint under some assumptions for the probability distribution of the weights. The problem becomes a second-order cone-constrained binary knapsack problem, which is equivalent to a robust binary knapsack problem with an ellipsoidal uncertainty set. 

We demonstrate that optimal solutions to robust binary knapsack problems with inner and outer polyhedral approximations of the ellipsoidal uncertainty set can provide both upper and lower bounds on the optimal value of the second-order cone-constrained binary knapsack problem, and they can be obtained by solving ordinary binary knapsack problems repeatedly. Moreover, we prove that the solution providing the upper bound converges to the optimal solution to the second-order cone-constrained binary knapsack problem as the approximation of the uncertainty set becomes more accurate. Based on this, a pseudo-polynomial time algorithm is obtained under the assumption that the coefficients of the problem are integer-valued. We also propose an exact algorithm, which iteratively improves the accuracy of the approximation until an exact optimal solution is obtained. The exact algorithm also runs in pseudo-polynomial time. Computational results are presented to show the quality of the upper and lower bounds and efficiency of the exact algorithm compared to CPLEX and a previous study.
\keywords{Chance-constrained binary knapsack problem \and Nonlinear combinatorial optimization \and Integer second-order cone-constrained programming \and Stochastic programming \and distributionally robust optimization}
\end{abstract}

\section{Introduction} \label{sect01}

The binary knapsack problem (BKP) is one of the most widely studied combinatorial optimization problems. The problem consists of a knapsack with a nonnegative integer capacity $b$ and a set of items ${N}:=\{1,\cdots,n\}$. Each item $j \in {N}$ has a nonnegative integer profit $p_{j}$ and a nonnegative integer weight $a_{j}$. The objective is to select the items that maximize the profit sum while satisfying the knapsack capacity constraint. There has been extensive research on the BKP and its variants due to their importance in theory and applications. We refer to Martello and Toth \cite{Martello90} and Kellerer et al. \cite{Kellerer04} for more details on the BKP. 

Here, we consider the BKP with uncertainty in the item weights. Optimization problems with uncertainty in data have drawn considerable attention recently. Stochastic programming \cite{Shapiro09,Birge11} and robust optimization \cite{Bental09,Bertsimas11} are two major approaches taken to handle data uncertainty in optimization problems. In stochastic programming, each uncertain coefficient is assumed to be a random variable with some specified probability distribution. The stochastic BKP with simple recourse \cite{Cohn98,Fortz08,Kosuch10,Merzifonluoglu12} reflects the expected value of penalty costs incurred by the total weight exceeding the capacity to the objective function. 

Chance-constrained programming is a class of stochastic programming, which finds a solution that satisfies the constraints with a probability of at least a given threshold value. The chance-constrained binary knapsack problem (CKP) has also been considered by many researchers. Kleinberg et al. \cite{Kleinberg00} presented polynomial-time algorithms for the CKP with Bernoulli distributed weights under some conditions. Goel and Indyk \cite{Goel99} developed polynomial-time approximation schemes (PTAS) for the problem with Bernoulli and exponentially distributed weights and a quasi-PTAS for the CKP with Bernoulli distributed weights. Goyal and Ravi \cite{Goyal10} proposed a PTAS by solving the reformulated two-dimensional BKP and rounding its solution for the CKP with independent normally distributed weights. Recently, De \cite{De18} proposed fully polynomial-time approximation schemes (FPTAS) for the CKP with Bernoulli distributed weights. A PTAS for the CKP with random weights whose variation and kurtosis are bounded was also designed by De \cite{De18}.

In robust optimization \cite{Bental09,Bertsimas11}, an uncertainty set containing all possible values of uncertain coefficients is defined. A robust optimal solution should satisfy the  constraints for any possible value of data in the uncertainty set. The robust BKP with a cardinality-constrained uncertainty set \cite{Bertsimas04,Lee12,Monaci13_2,Monaci13} has been investigated to obtain robust solutions efficiently. Klopfenstein and Nace \cite{Klopfenstein08} used solutions of the robust BKP with cardinality-constrained uncertainty set to produce feasible solutions to the CKP without an assumption regarding the distributions of weights. Moreover, Han et al. \cite{Han16} dealt with the robust BKP with inner approximation of the ellipsoidal uncertainty set to obtain qualified approximate solutions to the CKP with independent normally distributed weights. 

distributionally robust optimization \cite{Rahimian19} has also been proposed as a generalization of previous approaches. In distributionally robust optimization, uncertain coefficients are assumed to be random variables as in stochastic programming, but they can have any distribution belonging to a specified ambiguity set. Recently, several variants of the distributionally robust BKP, which are generalizations of the stochastic BKP, have been considered. Cheng et al. \cite{Cheng14} considered the distributionally robust quadratic BKP only with the first and second moment of the random parameters available, and semidefinite programming reformulation was applied to obtain the upper and lower bounds for the problem. Bansal et al. \cite{Bansal18} applied Benders' decomposition and the L-shaped method for the distributionally robust multiple binary knapsack problem and other variants. Xie \cite{Xie19} presented conditional value-at-risk constrained reformulations of a distributionally robust optimization problem with chance constraints under the Wasserstein ambiguity set, and their efficiency was tested for the distributionally robust multidimensional knapsack problem.  

In this paper, we assume that weight of each item is a random variable. Then, the deterministic capacity constraint is converted to the following chance constraint. 
\begin{equation} \label{eq01_01} 
\textnormal{Prob} \left\{ \sum_{j \in {N}} a_{j} x_{j} \leq b \right\} \geq \rho,
\end{equation}
where $\rho \geq 0.5$ is the confidence level of the chance constraint. The CKP now can be formulated using constraint (\ref{eq01_01}) as
\begin{align*} 
\textbf{(CKP)} \ \textnormal{maximize} & \ \ \sum_{j \in {N}} p_{j} x_{j} \\
\textnormal{subject to} & \ \ \textnormal{Prob} \left\{ \sum_{j \in {N}} a_{j} x_{j} \leq b \right\} \geq \rho, \\
& \ \ x \in \{0,1\}^{n}. 
\end{align*}
When the weights are normally distributed with the means and covariances known, the CKP can be represented as a deterministic second-order cone (SOC)-constrained BKP \cite{Calafiore06}. It is also known that the problem is equivalent to a robust BKP with an ellipsoidal uncertainty set \cite{Bental99}. Moreover, if the weights are independent, the left-hand side of the SOC constraint can be represented as the sum of a linear function and the root of another linear function, which is a submodular function \cite{Atamturk09}.

Several researchers have attempted to solve the SOC-constrained BKP with the capacity constraint of which the left-hand side is a submodular nonlinear function. As previously mentioned, Goyal and Ravi \cite{Goyal10} developed a PTAS for the problem, and they also provided an example in which the integrality gap of the conic relaxation is at least $\mathrm{\Omega}(\sqrt{n})$. Han et al. \cite{Han16} proposed an algorithm for the robust BKP with an ellipsoidal uncertainty set using an inner polyhedral approximation of the ellipsoidal uncertainty set. Their algorithm provides an approximate solution to the CKP, which has the theoretical probability guarantee of satisfying constraint (\ref{eq01_01}) with probability $\rho-\delta$ for any fixed positive $\delta$. Based on the approach of Nikolova \cite{Nikolova10}, Shabtai et al. \cite{Shabtai18} presented a relaxed FPTAS with the running time bounded by a polynomial function of $n$ and $\epsilon^{-1}$. Joung and Lee \cite{Joung20} used submodularity to develop a heuristic algorithm, which has a procedure of solving the robust BKP with cardinality-constrained uncertainty set \cite{Bertsimas03,Bertsimas04} repeatedly while increasing the degree of robustness. Yang and Charkraborty \cite{Yang20} developed an exact algorithm, which solves the ordinary BKPs iteratively. Their ordinary problems have a capacity constraint in which the left-hand side is the convex combination of the total mean and the total variation. Their algorithm worked well for some instances with a maximum of 100 items that they generated, but they did not provide any theoretical guarantee for the efficiency of their algorithm. 

In distributionally robust optimization, the chance constraint should be  satisfied when the random vector $a$ follows any probability distribution in the ambiguity set $\mathcal{D}_{a}$. Then, the chance constraint also can be converted to the following distributionally robust chance constraint: 
\begin{equation} \label{eq01_02} 
\mathrm{inf}_{a \in \mathcal{D}_{a}} \ \textnormal{Prob} \left\{ \sum_{j \in {N}} a_{j} x_{j} \leq b \right\} \geq \rho.
\end{equation}
The distributionally robust chance-constrained BKP (DCKP) can be obtained by replacing constraint (\ref{eq01_01}) with constraint (\ref{eq01_02}) as 
\begin{align*} 
\textbf{(DCKP)} \ \textnormal{maximize} & \ \ \sum_{j \in {N}} p_{j} x_{j} \\
\textnormal{subject to} & \ \ \mathrm{inf}_{a \in \mathcal{D}_{a}} \ \textnormal{Prob} \left\{ \sum_{j \in {N}} a_{j} x_{j} \leq b \right\} \geq \rho, \\
& \ \ x \in \{0,1\}^{n}. 
\end{align*}
\noindent To the best of our knowledge, little research has focused on the DCKP and its solution intensively, although some of its variants have been sometimes considered.

We consider theoretical investigation and algorithms for chance-constrained BKPs with uncertainty in item weights. Depending on the assumption regarding the distribution of random weights, the capacity constraint can be represented as the chance constraint (\ref{eq01_01}) or the distributionally robust chance constraint (\ref{eq01_02}), and then \textbf{(CKP)} or \textbf{(DCKP)} can be obtained, respectively. 

This paper is organized as follows. Section \ref{sect02} interprets the relationships between chance constraints with various assumptions on the distribution of item weights and the SOC constraint. In section \ref{sect03}, we associate the SOC constraint with the equivalent robust constraint with an ellipsoidal uncertainty set. We then present inner and outer polyhedral approximations of the ellipsoidal uncertainty set. Section \ref{sect04} presents an algorithm that solves ordinary BKP repeatedly to obtain the upper and lower bounds for the SOC-constrained BKP. Section \ref{sect05} explains some probability guarantees of the solutions obtained by the algorithm. Section \ref{sect06} describes a pseudo-polynomial time exact algorithm when the coefficients are integer-valued and an iterative algorithm to efficiently obtain an exact optimal solution for the SOC-constrained BKP. Section \ref{sect07} gives computational results of our approximate and exact algorithms compared to CPLEX and a previous study. Section \ref{sect08} summarizes the main results of our research and proposes suggestions for future research.   

\section{chance constraints and second-order cone constraint} \label{sect02}

Let $a$ be an $n$-dimensional random vector with a mean vector $\hat{a}$ and a positive semidefinite covariance matrix $\mathrm{\Sigma}$. There is a well-known reformulation for stochastic programming problems that include chance constraints (\ref{eq01_01}). When each $a_{j}$ has normal distribution, constraint (\ref{eq01_01}) can be reformulated as \cite{Calafiore06}
\begin{equation} \label{eq02_01}
\hat{a}' x + \mathrm{\Phi}^{-1}(\rho) \ \| \mathrm{\Sigma}^{1/2} x \| \leq b,
\end{equation}
where $\mathrm{\Phi}$ is the cumulative distribution function of the standard normal distribution. This reformulation is called the second-order cone constraint.  

We can also derive similar reformulations for distributionally robust optimization problems that include distributionally robust chance constraints (\ref{eq01_02}). Let $\mathcal{D}_{a}^{1}$ be a set of arbitrary random vectors with mean vector $\hat{a}$ and covariance matrix $\mathrm{\Sigma}$, that is, $\mathcal{D}_{a}^{1} := \{a \ | \ \textrm{E}[a]=\hat{a}, \textrm{E}[(a-\hat{a})(a-\hat{a})'] =  \mathrm{\Sigma}  \}$. Then, the capacity constraint becomes distributionally robust chance constraint (\ref{eq01_02}) with $\mathcal{D}_{a} = \mathcal{D}_{a}^{1}$, which can be reformulated as \cite{Calafiore06}
\begin{equation} \label{eq02_02}
\hat{a}' x + \sqrt{\frac{\rho}{1-\rho}} \ \| \mathrm{\Sigma}^{1/2} x \| \leq b.
\end{equation}
The one-tailed version of Chebyshev's inequality \cite{Pinter89}, also called Cantelli's inequality, can be used to reformulate constraint (\ref{eq01_02}) with $\mathcal{D}_{a} = \mathcal{D}_{a}^{1}$ to constraint (\ref{eq02_02}). 

Moreover, Delage and Ye \cite{Delage10} proposed a new ambiguity set to avoid the difficulty that arises when the expected values and covariance matrix obtained from data samples are not independent from estimation errors. When a vector of sample average value $\hat{a}$ and a sample covariance matrix $\hat{\mathrm{\Sigma}}$ are given, the new uncertainty set $\mathcal{D}_{a}^{2}$ can be expressed as
\begin{equation} \nonumber
\mathcal{D}_{a}^{2} := \{ a \ | \ (\textrm{E}[a]-\hat{a})' \hat{\mathrm{\Sigma}}^{-1} (\textrm{E}[a]-\hat{a}) \leq \gamma_{1}, \textrm{E}[(a-\hat{a})(a-\hat{a})'] \succeq \gamma_{2} \hat{\mathrm{\Sigma}} \},
\end{equation}
where $\gamma_{1} \geq 0$ and $\gamma_{2} \geq 1$. Here, $A \succeq B$ means that $A - B$ is positive semidefinite for square matrices $A$ and $B$. Zhang et al. \cite{Zhang18} proved that constraint (\ref{eq01_02}) with $\mathcal{D}_{a} = \mathcal{D}_{a}^{2}$ can be reformulated as an SOC constraint under certain conditions. If $\gamma_{1}/\gamma_{2} \leq 1-\rho$, constraint (\ref{eq01_02}) with $\mathcal{D} = \mathcal{D}_{a}^{2}$ is equivalent to 
\begin{equation} \nonumber
\hat{a}' x + (\sqrt{\gamma_{1}} + \sqrt{(\gamma_{2}-\gamma_{1})\frac{\rho}{1-\rho}}) \ \| \mathrm{\Sigma}^{1/2} x \| \leq b,
\end{equation}
and if $\gamma_{1}/\gamma_{2} > 1-\rho$, the constraint is equivalent to
\begin{equation} \label{eq02_04}
\hat{a}' x + \sqrt{\frac{\gamma_{2}}{1-\rho}} \ \| \mathrm{\Sigma}^{1/2} x \| \leq b.
\end{equation}
 
On the other hand, there can be some information about the lower and upper bounds on the values of the coefficients although the exact values are not known. Then, we can reflect it in an ambiguity set. Let $a_{j}$ be a random variable with a mean $\hat{a}_{j}$ and its support interval $[\underline{a}_{j},\bar{a}_{j}]$ for $j \in {N}$. Calafiore and Ghaoui \cite{Calafiore06} defined ambiguity set $\mathcal{D}_{a}^{3} := \{ a \ | \ \textrm{E}[a]=\hat{a}, \ \underline{a} \leq a \leq \bar{a}\}$, and verified that constraint (\ref{eq01_02}) with $\mathcal{D}_{a} = \mathcal{D}_{a}^{3}$ holds if the following constraint is alternatively available:
\begin{equation} \label{eq02_05}
\hat{a}' x + \sqrt{-\frac{1}{2}\mathrm{ln}(1-\rho)} \ \sqrt{\sum_{j \in {N}} (\bar{a}_{j}^{2} - \underline{a}_{j}^{2})^{2} x_{j}^{2}} \leq b.
\end{equation}

Constraints (\ref{eq02_01}) to (\ref{eq02_05}) are SOC constraints. A continuous optimization problem with a linear objective function and linear and SOC constraints is an SOC-constrained programming problem. It is an important class of convex optimization problems, which can be solved in polynomial time. In contrast, if an optimization problem with SOC constraints has discrete variables, the problem becomes a mixed integer nonlinear programming problem, for which there has been limited success in finding solutions \cite{Benson13}.

We can observe that several variants of stochastic programming and distributionally robust optimization problems can be reformulated as equivalent or safe-approximate deterministic optimization problems with SOC constraints as constraints (\ref{eq02_01}) to (\ref{eq02_05}). We also assume that each weight $a_{j}$ of item $j$ is independent of the weights of other items. Under this assumption, the covariance matrix $\mathrm{\Sigma}$ becomes a diagonal matrix, that is, $\mathrm{\Sigma} := \textrm{diag}(\sigma_{1}^{2},\cdots,\sigma_{n}^{2})$, where $\sigma_{j}$ is the standard deviation of $a_{j}$ for item $j \in {N}$. Then, $\| \mathrm{\Sigma}^{1/2} x \|$ becomes $\sqrt{ \sum_{j \in {N}} {\sigma}_{j}^{2} x_{j}^{2} }$, and we can express constraints (\ref{eq02_01}) to (\ref{eq02_04}) as the following formulation using some $\mathrm{\Omega}$:
\begin{equation} \label{eq02_06}
\sum_{j \in {N}} \hat{a}_{j} x_{j} + \mathrm{\Omega} \sqrt{ \sum_{j \in {N}} {\sigma}_{j}^{2} x_{j}^{2} } \leq b.
\end{equation}
In addition, the problem we consider becomes the following SOC constrained BKP which includes constraint (\ref{eq02_06}):
\begin{align*} 
\textbf{(SOCKP)} \ \textnormal{maximize} & \ \ \sum_{j \in {N}} p_{j} x_{j} \\
\textnormal{subject to} & \ \ \sum_{j \in {N}} \hat{a}_{j} x_{j} + \mathrm{\Omega} \sqrt{ \sum_{j \in {N}} {\sigma}_{j}^{2} x_{j}^{2} } \leq b, \\
& \ \ x \in \{0,1\}^{n}. 
\end{align*}

\section{Approximations of the second-order cone constraint via robust optimization} \label{sect03}

For the approximation of the SOC constraint (\ref{eq02_06}), we first consider the following equivalent robust constraint with an ellipsoidal uncertainty set \cite{Bental99}: 
\begin{equation} \label{eq03_00} 
\sum_{j \in N} a_{j} x_{j} \leq b, \ \ \forall a \in \mathcal{U},
\end{equation}
where $\mathcal{U} := \{ \hat{a} \ + \  \mathrm{\Sigma}^{1/2} \xi \ | \ \| \xi \| \leq \mathrm{\Omega} \}$ is the ellipsoidal uncertainty set. The ellipsoidal uncertainty set $\mathcal{U}$ is equal to $\mathcal{U}_{\mathrm{\Omega}}$, where 
\begin{align} \nonumber
\mathcal{U}_{\gamma} := \{ \hat{a} + \mathrm{\Sigma}^{1/2} \xi \ | \ \sum_{j \in {N}} \xi_{j}^{2} \leq \gamma^{2} \}, \ \gamma > 0.
\end{align}
To approximate the uncertainty set $\mathcal{U}$, we use piecewise-linear approximations of the $n$ quadratic functions $q_{j}(\xi_{j}):=\xi_{j}^{2}$, $j \in {N}$ with the identical domains $\xi_{j} \in [0,\mathrm{\Omega}]$, $j \in {N}$. 

For the upper approximation of the quadratic function, the domain $[0,\mathrm{\Omega}]$ of $q_{j}(\xi_{j})$ is divided into $m$ equally sized subintervals $[\pi_{j}^{0},\pi_{j}^{1}]$, $\cdots$, $[\pi_{j}^{m-1},\pi_{j}^{m}]$ along the horizontal axis  for each $j \in {N}$, where $\pi_{j}^{k} := (k / m) \cdot \mathrm{\Omega}$, $k=0,\cdots,m$, $j \in {N}$. Then, a piecewise linear function $u^{m}_{j}: [0,\mathrm{\Omega}] \to [0,\mathrm{\Omega}^2]$ whose $k$-th linear segment is over the subinterval $[\pi_{j}^{k-1},\pi_{j}^{k}]$ can be defined as
\begin{eqnarray}
u^{m}_{j}(\xi_{j}) & := & (\pi^{k-1}_{j})^2 + \frac{(\pi^{k}_{j})^2-(\pi^{k-1}_{j})^2}{\mathrm\pi_{j}^{k}-\pi_{j}^{k-1}} \ (\xi_{j} - \pi_{j}^{k-1}) \nonumber \\ 
& = & (\pi_{j}^{k-1}+\pi_{j}^{k}) \cdot \xi_{j} - \pi_{j}^{k-1}\pi_{j}^{k} \nonumber \\
& = & \frac{(2k-1) }{m} \mathrm{\Omega} \cdot \xi_{j} - \frac{(k-1)k }{m^2} \mathrm{\Omega}^{2}, \ \textrm{if} \ \pi_{j}^{k-1} \leq \xi_{j} \leq  \pi_{j}^{k}, \ \forall k \in {M}, \nonumber
\end{eqnarray}
where ${M} := \{1,\cdots,m\}$ is an index set of the linear segments. The piecewise linear function $u^{m}_{j}$ approximates $q_{j}$ from above, that is, $u^{m}_{j}(\xi_{j}) \geq q_{j}(\xi_{j})$, $\xi_{j} \in [0, \mathrm{\Omega}]$ for $j \in {N}$. Figure \ref{fig03_01} shows examples of $u^{m}_{j}$ for $m=2$ and $m=5$ when $\mathrm{\Omega} = 10$. 

As the number of linear segments $m$ increases in the approximation, the approximation becomes more accurate. Moreover, we can identify the maximum error of the approximation depending on $m$ as the following lemma shows. 

\begin{lemma} \label{lemma03_01}  
The maximum difference between $u^{m}_{j}$ and the quadratic function $q_{j}$ is $\mathrm{\Omega}^{2}/4 m^{2}$, that is, 
\begin{equation} \nonumber
u^{m}_{j}(\xi_{j}) \leq q_{j}(\xi_{j}) + \frac{\mathrm{\Omega}^{2}}{4 m^{2}}, \ \xi_{j} \in [0,\mathrm{\Omega}], \ \forall j \in {N}.
\end{equation}
Moreover, the bound is tight at the center of each of the subintervals $[\pi_{j}^{k-1},\pi_{j}^{k}]$, $k \in {M}$. 
\end{lemma}
\begin{proof}
Let $g^{m}_{j}$ be the difference between $u^{m}_{j}$ and $q_{j}$, that is, $g^{m}_{j}(\xi_{j}) := u^{m}_{j}(\xi_{j}) - q_{j}(\xi_{j})$, $\xi_{j} \in [0,\mathrm{\Omega}]$. Then, $g^{m}_{j}$ is a nonnegative and continuous piecewise quadratic function, and it can be represented by 
\begin{equation*} 
g^{m}_{j}(\xi_{j}) = -\xi_{j}^{2} + \frac{2k-1}{m} \mathrm{\Omega} \cdot \xi_{j} - \frac{(k-1)k}{m^{2}} \mathrm{\Omega}^{2}, \ \textrm{if} \  \pi_{j}^{k-1} \leq \xi \leq \pi_{j}^{k}, \ \forall k \in {M}.
\end{equation*}
Then, $g^{m}_{j}$ can be differentiated at the domain of the function except the endpoints of the linear segments, and its derivative over each of the open intervals $(\pi_{j}^{k-1},\pi_{j}^{k})$, $k \in {M}$ is a linear function represented by
\begin{equation*} 
\frac{d g^{m}_{j}(\xi_{j})}{d \xi_{j}} = -2 \xi_{j} + \frac{2k-1}{m} \mathrm{\Omega}, \ \textrm{if} \ \pi_{j}^{k-1} < \xi_{j} < \pi_{j}^{k}, \ \forall k \in {M}.
\end{equation*}
The local maximum points of $g^{m}_{j}$, $k \in {M}$ are $\xi_{j,k}^{*} := (2k-1) \mathrm{\Omega}/2m$, $k \in {M}$, obtained from ${d g_{m}(\xi_{j,k}^{*})}/{d \xi_{j}}=0$. These points, which are the centers of the subintervals, have the same function value of $g_{j}^{m}(\xi_{j,k}^{*}) = \mathrm{\Omega}^{2} / {4m^{2}}$ for all $k \in {M}$. \qed
\end{proof}

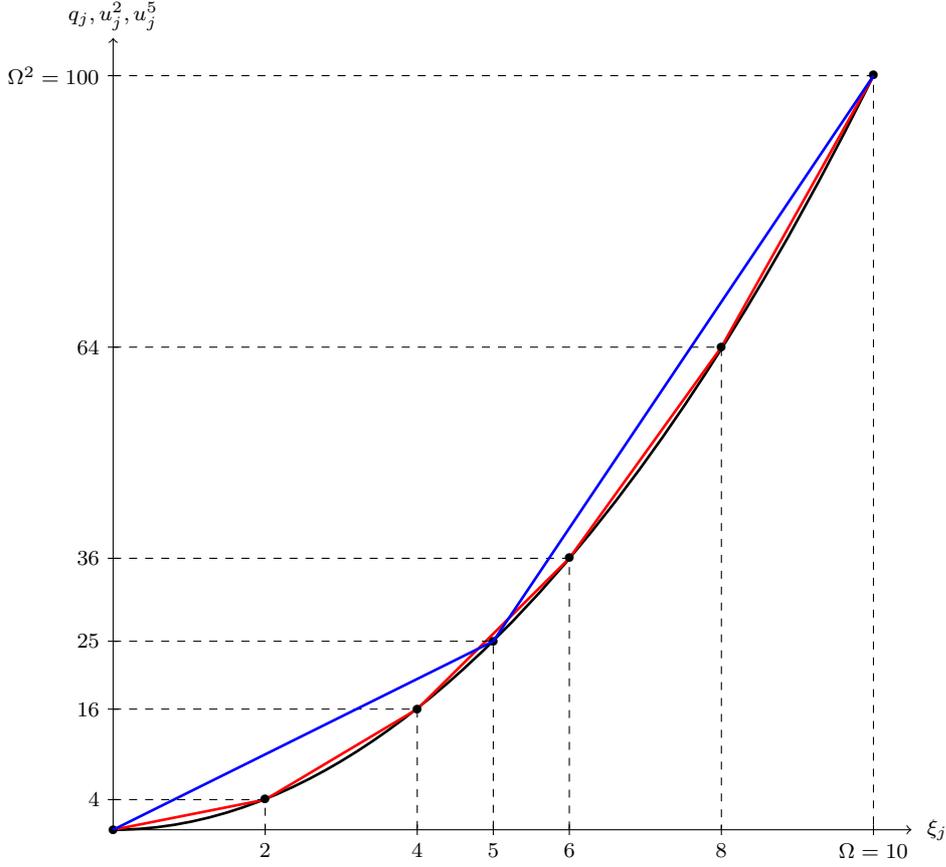
\begin{figure}
\begin{center}
\begin{tikzpicture}[scale=10]
\draw[->] (-0,0) -- (1.05,0) node[right] {$\ \xi_j$};
\draw[->] (0,-0) -- (0,1.05) node[above] {$q_{j}, u^{2}_{j}, u^{5}_{j}$};
\foreach \x/\xtext in {0.2/{2}, 0.4/4, 0.5/5, 0.6/6, 0.8/8, 1.0/\mathrm{\Omega}=10}
\draw[shift={(\x,0)}] (0pt,0.2pt) -- (0pt,-0.2pt) node[below] {$\xtext$};
\foreach \y/\ytext in {0.04/4, 0.16/16, 0.25/25, 0.36/36, 0.64/64, 1.0/\mathrm{\Omega}^{2}=100}
\draw[shift={(0,\y)}] (0.2pt,0pt) -- (-0.2pt,0pt) node[left] {$\ytext$};
\draw[line width=1pt] (0,0) parabola bend (0,0) (1,1);
\draw[line width=1pt, red] (0,0) -- (0.2,0.04) -- (0.4, 0.16) -- (0.6, 0.36) -- (0.8, 0.64) -- (1,1);
\draw (0,0) node[] {\textbullet};
\draw (0.2,0.04) node[] {\textbullet};
\draw (0.4, 0.16) node[] {\textbullet};
\draw (0.6, 0.36) node[] {\textbullet};
\draw (0.8, 0.64) node[] {\textbullet};
\draw (1,1) node[] {\textbullet};
\draw[dashed] (0,0.04) -- (0.2,0.04) -- (0.2,0);
\draw[dashed] (0, 0.16) -- (0.4, 0.16) -- (0.4, 0);
\draw[dashed] (0, 0.25) -- (0.5, 0.25) -- (0.5, 0);
\draw[dashed] (0, 0.36) -- (0.6, 0.36) -- (0.6, 0);
\draw[dashed] (0, 0.64) -- (0.8, 0.64) -- (0.8, 0);
\draw[dashed] (0,1) -- (1,1) -- (1,0);
\draw (0.5, 0.25) node[] {\textbullet};
\draw[line width=1pt, blue] (0,0) -- (0.5, 0.25)  -- (1,1);
\end{tikzpicture} 
\caption{The quadratic function $q_{j}$ (black) and its upper approximations $u^{2}_{j}$ (blue) and $u^{5}_{j}$ (red), respectively when $\mathrm{\Omega} = 10$.}
\label{fig03_01}
\end{center}
\end{figure}

We now obtain an inner approximation of the uncertainty set $\mathcal{U}$ by using the upper approximation of the quadratic functions $q_{j}$, $j \in {N}$. After replacing $q_{j}$ by $u_{j}^{m}$, $j \in {N}$, an uncertainty set $\mathcal{U}_{\textrm{in}}^{m}$ can be obtained as
\begin{equation} \nonumber
\mathcal{U}_{\textrm{in}}^{m} = \{ \hat{a} + \mathrm{\Sigma}^{1/2} \xi \ | \ \sum_{j \in {N}} u^{m}_{j}(\xi_{j}) \leq \mathrm{\Omega}^{2} \}.
\end{equation}
Because $u^{m}_{j}$ is an approximation of $q_{j}$ from above, we have  $\mathcal{U}_{\textrm{in}}^{m} \subseteq \mathcal{U}$. The inner approximation $\mathcal{U}_{\textrm{in}}^{m}$ of $\mathcal{U}$ improves quickly as the value of $m$ becomes larger. Moreover, $\mathcal{U}_{\textrm{in}}^{m}$ is a polyhedron because $u^{m}_{j}$ is a convex and piecewise linear function.  

The robust counterpart using $\mathcal{U}_{\textrm{in}}^{m}$ instead of $\mathcal{U}$ can be expressed as
\begin{equation} \label{eq03_05} 
\sum_{j \in {N}} a_{j} x_{j} \leq b, \ \ \forall a \in \mathcal{U}_{\textrm{in}}^{m}.
\end{equation}
Because $\mathcal{U}_{\textrm{in}}^{m} \subseteq \mathcal{U}$, an optimal solution to the robust BKP with uncertainty set $\mathcal{U}_{\textrm{in}}^{m}$ can be infeasible to the problem with the original constraint (\ref{eq03_00}) and \textbf{(SOCKP)}, even though the optimal value of the problem with constraint (\ref{eq03_05}) provides an upper bound on the optimal value of \textbf{(SOCKP)}. However, we can see that $\mathcal{U}_{\textrm{in}}^{m}$ becomes close to $\mathcal{U}$ as the value of $m$ increases. Suppose that $\mathcal{U}_{{\gamma}_{m}} = \{ \hat{a} \ + \  \mathrm{\Sigma}^{1/2} \xi \ | \ \| \xi \| \leq {\gamma}_{m} \} $ is an ellipsoid contained in $\mathcal{U}_{\textrm{in}}^{m}$. Then, we have the following proposition.  

\begin{proposition} \label{prop03_01}
Let $\bar{\gamma}_{m}$ be the maximum value of $\gamma_{m}$ such that $\mathcal{U}_{\gamma_{m}} \subseteq \mathcal{U}_{\textrm{in}}^{m}$. Then, 
\begin{equation*} 
\bar{\gamma}_{m} \geq \mathrm{\Omega} \sqrt{1 - \frac{n}{4m^2}},
\end{equation*}
when $m \geq \sqrt{n} / 2$. Moreover,
\begin{equation*}  
\lim_{m \to \infty} \bar{\gamma}_{m} = \mathrm{\Omega}.
\end{equation*}
\end{proposition}
\begin{proof}
From Lemma \ref{lemma03_01}, the following inequalities hold; 
\begin{equation} \label{eq03_06} 
u_{j}^{m}({\xi}_{j}) \leq {\xi}_{j}^{2} + \frac{\mathrm{\Omega}^{2}}{4 m^{2}}, \ \forall j \in {N}.  
\end{equation}
Moreover, the following inequality derived from adding inequalities (\ref{eq03_06}) for all $j \in {N}$ holds:
\begin{equation*} 
\sum_{j \in {N}} u_{j}^{m}({\xi}_{j}) \leq \|{\xi}\|^{2} + n \frac{\mathrm{\Omega}^{2}}{4 m^{2}} \leq \mathrm{\Omega}^{2}, 
\end{equation*}
when $\|{\xi}\| \leq \mathrm{\Omega} \sqrt{1 - {n} / {4m^{2}}}$, $m \geq \sqrt{n} / 2$. Therefore, $\mathcal{U}_{{\gamma}_{m}}$ with ${\gamma}_{m} = \mathrm{\Omega} \sqrt{1 - {n} / {4m^{2}}}$ has $\mathcal{U}_{{\gamma}_{m}} \subseteq \mathcal{U}_{\textrm{in}}^{m}$, when $m \geq \sqrt{n} / 2$. This implies that the value of $\bar{\gamma}_{m}$ is at least $\mathrm{\Omega} \sqrt{1 - {n}/{4m^2}}$. The second statement also holds because $\sqrt{1 - {n}/{4m^2}} \to 1$ as $m \to \infty$. \qed 
\end{proof}

The derivation of the inner polyhedral approximation of the ellipsoidal uncertainty set $\mathcal{U}$ has already been attempted by Han et al. \cite{Han16}. They divided the interval $[0,\mathrm{\Omega}^{2}]$ in the range of the function $q_{j}$ (vertical axis) into $m$ equal-sized subintervals to obtain a piecewise linear approximation of $q_{j}$ from above. However, we divide the interval $[0, \mathrm{\Omega}]$ in the domain of the function $q_{j}$ (horizontal axis) into $m$ equal-sized subintervals to obtain a piecewise linear approximation, which results in more accurate approximation of $q_{j}$. By Lemma \ref{lemma03_01}, the difference between our piecewise linear approximation $u_{j}^{m}$ and $q_{j}$ is at most $\mathrm{\Omega}/4m^{2}$, while the maximum difference between Han et al.'s approximation and $q_{j}$ is $\mathrm{\Omega}^{2}/4m$. Hence, our scheme results in more accurate approximation of the ellipsoidal uncertainty set $\mathcal{U}$. 

Next, we consider the lower piecewise linear approximation of the quadratic functions $q_{j}$, $j \in {N}$. By using our scheme, we can obtain a piecewise linear approximation of $q_{j}$ from below too, which Han et al. \cite{Han16} did not consider. Thus, we can obtain both inner and outer polyhedral approximation of $\mathcal{U}$, which can help us obtain good algorithms for \textbf{(SOCKP)}. We define new piecewise linear functions $l^{m}_{j}$, $j \in {N}$ by subtracting $\mathrm{\Omega}^{2}/4m^{2}$ from each $u^{m}_{j}$ function, that is, $l^{m}_{j} := u^{m}_{j} - \mathrm{\Omega}^{2}/4m^{2} $, $j \in {N}$. Figure \ref{fig03_02} illustrates examples of $l^{m}_{j}$ for $m=2$ and $m=5$ when $\mathrm{\Omega} = 10$.  From Lemma \ref{lemma03_01}, $l^{m}_{j}$ is a piecewise linear approximation of $q_{j}$ from below. Moreover, the following lemma shows the accuracy of the lower approximation. 

\begin{lemma} \label{lemma03_02}
The maximum difference between $l^{m}_{j}$ and the quadratic function $q_{j}$ is $\mathrm{\Omega}^{2}/4 m^{2}$, that is, 
\begin{equation} \nonumber
l^{m}_{j}(\xi_{j}) \geq q_{j}(\xi_{j}) - \frac{\mathrm{\Omega}^{2}}{4 m^{2}}, \ \xi_{j} \in [0,\mathrm{\Omega}], \ \forall j \in {N}.
\end{equation}
Moreover, the bound is tight at $\xi_{j} = \pi_{j}^{0}, \pi_{j}^{1},\cdots,\pi_{j}^{m}$. 
\end{lemma}
\begin{proof}
Because $u^{m}_{j}$ is a piecewise linear approximation of of $q_{j}$ from above, $u^{m}_{j}(\xi_{j}) - q_{j}(\xi_{j}) = l^{m}_{j}(\xi_{j}) + \mathrm{\Omega}^{2}/4m^{2} - q_{j}(\xi_{j}) \geq 0$ is satisfied for $\xi_{j} \in [0,\mathrm{\Omega}]$. $l^{m}_{j}(\xi_{j}) = u^{m}_{j}(\xi_{j}) - \mathrm{\Omega}^{2}/4m^{2}$ is equal to $q_{j}(\xi_{j}) - \mathrm{\Omega}^{2}/4m^{2}$ at $\xi_{j}=\pi_{j}^{0}, \pi_{j}^{1}, \cdots, \pi_{j}^{m}$.  \qed
\end{proof}

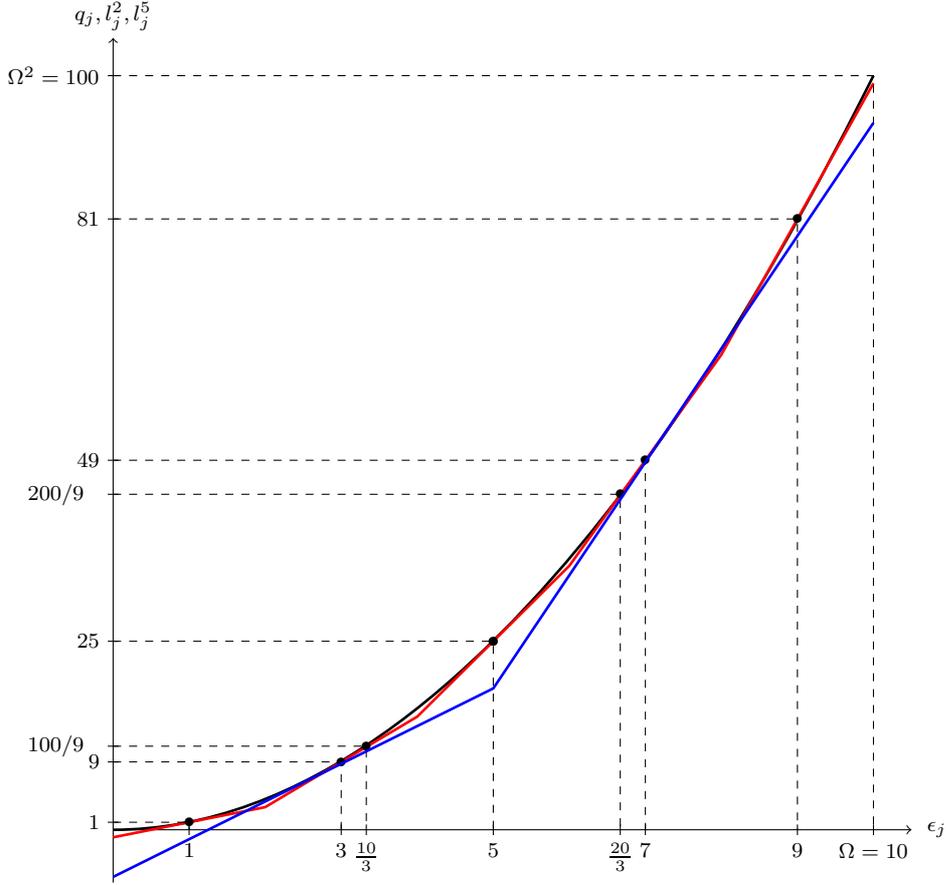
\begin{figure}
\begin{center}
\begin{tikzpicture}[scale=10]
\draw[->] (-0,0) -- (1.05,0) node[right] {$\ \epsilon_j$};
\draw[->] (0,-0.07) -- (0,1.05) node[above] {$q_{j}, l^{2}_{j}, l^{5}_{j}$};
\foreach \x/\xtext in {0.1/1, 0.3/3, 0.333/\frac{10}{3}, 0.5/5, 0.667/\frac{20}{3}, 0.7/7, 0.9/9, 1.0/\mathrm{\Omega}=10}
\draw[shift={(\x,0)}] (0pt,0.2pt) -- (0pt,-0.2pt) node[below] {$\xtext$};
\foreach \y/\ytext in {0.01/1, 0.09/9, 0.333*0.333/{100\mathrm{/}9 \ \ }, 0.25/25,  0.49/49, 0.667*0.667/{200\mathrm{/}9 \ \ }, 0.81/81, 1.0/\mathrm{\Omega}^{2}=100}
\draw[shift={(0,\y)}] (0.2pt,0pt) -- (-0.2pt,0pt) node[left] {$\ytext$};
\draw[line width=1pt] (0,0) parabola bend (0,0) (1,1);
\draw[line width=1pt, red] (0,0-0.01) -- (0.2,0.04-0.01) -- (0.4, 0.16-0.01) -- (0.6, 0.36-0.01) -- (0.8, 0.64-0.01) -- (1,1-0.01);
\draw (0.1, 0.01) node[] {\textbullet};
\draw (0.3, 0.09) node[] {\textbullet};
\draw (0.5, 0.25) node[] {\textbullet};
\draw (0.7, 0.49) node[] {\textbullet};
\draw (0.9, 0.81) node[] {\textbullet};
\draw (0.333,0.333*0.333) node[] {\textbullet};
\draw (0.667,0.667*0.667) node[] {\textbullet};
\draw[dashed] (0, 0.01) -- (0.1, 0.01) -- (0.1, 0);
\draw[dashed] (0, 0.09) -- (0.3, 0.09) -- (0.3, 0);
\draw[dashed] (0, 0.25) -- (0.5, 0.25) -- (0.5, 0);
\draw[dashed] (0, 0.49) -- (0.7, 0.49) -- (0.7, 0);
\draw[dashed] (0, 0.81) -- (0.9, 0.81) -- (0.9, 0);
\draw[dashed] (0, 0.333*0.333) -- (0.333, 0.333*0.333) -- (0.333, 0);
\draw[dashed] (0, 0.667*0.667) -- (0.667, 0.667*0.667) -- (0.667, 0);
\draw[dashed] (0,1) -- (1,1) -- (1,0);
\draw (0.5, 0.25) node[] {\textbullet};
\draw[line width=1pt, blue] (0,-0.0625) -- (0.5, 0.25-0.0625)  -- (1,1-0.0625);
\end{tikzpicture} 
\caption{The quadratic function $q_{j}$ (black) and its lower approximations $l^{2}_{j}$ (blue) and $l^{5}_{j}$ (red), respectively when $\mathrm{\Omega} = 10$.}
\label{fig03_02}
\end{center}
\end{figure}

As the derivation of the inner approximation of $\mathcal{U}$, we can obtain an outer approximation of $\mathcal{U}$ using the lower approximation of the quadratic functions $q_{j}$, $j \in {N}$. After the quadratic functions are replaced by the $n$ piecewise linear functions $l^{m}_{j}$, $j \in {N}$, the outer approximate uncertainty set $\mathcal{U}_{\textrm{out}}^{m}$ can be expressed as
\begin{equation} \nonumber
\mathcal{U}_{\textrm{out}}^{m} = \{ \hat{a} + \mathrm{\Sigma}^{1/2} \xi \ | \ \sum_{j \in {N}} l^{m}_{j}(\xi_{j}) \leq \mathrm{\Omega}^{2} \}.
\end{equation}
Because $l^{m}_{j}$ is an approximation of $q_{j}$ from below, we have $\mathcal{U}_{\textrm{out}}^{m} \supseteq \mathcal{U}$. Furthermore, $\mathcal{U}_{\textrm{out}}^{m}$ is also a polyhedron, and the outer approximation $\mathcal{U}_{\textrm{out}}^{m}$ of $\mathcal{U}$ improves quickly as the value of $m$ becomes larger.

The robust counterpart using $\mathcal{U}_{\textrm{out}}^{m}$ instead of $\mathcal{U}$ can be expressed as
\begin{equation} \label{eq03_09} 
\sum_{j \in {N}} a_{j} x_{j} \leq b, \ \ \forall a \in \mathcal{U}_{\textrm{out}}^{m}.
\end{equation}
Because $\mathcal{U}_{\textrm{out}}^{m} \supseteq \mathcal{U}$, an optimal solution to the robust BKP with uncertainty set $\mathcal{U}_{\textrm{out}}^{m}$ is a feasible solution to the problem with the original constraint (\ref{eq03_00}), and \textbf{(SOCKP)} as well. Its optimal value provides a lower bound on the optimal value of \textbf{(SOCKP)}. Moreover, we can see that $\mathcal{U}_{\textrm{out}}^{m}$ becomes close to $\mathcal{U}$ as $m$ increases. Suppose that $\mathcal{U}_{{\gamma}_{m}}$ is an ellipsoid containing $\mathcal{U}_{\textrm{out}}^{m}$. Then, we have the following proposition. 

\begin{proposition} \label{prop03_02}
Let $\underline{\gamma}_{m}$ be the minimum value of $\gamma_{m}$ such that $\mathcal{U}_{\gamma_{m}}  \supseteq \mathcal{U}_{\textrm{out}}^{m}$. Then, 
\begin{equation*} 
\underline{\gamma}_{m} \leq \mathrm{\Omega} \sqrt{1 + \frac{n}{4m^2}}.
\end{equation*}
Moreover,
\begin{equation*}  
\lim_{m \to \infty} \underline{\gamma}_{m} = \mathrm{\Omega}.
\end{equation*}
\end{proposition}
\begin{proof}
From Lemma 2, the following inequalities hold: 
\begin{equation} \label{eq03_07} 
 {\xi}_{j}^{2}  \leq l_{j}^{m}({\xi}_{j}) + \frac{\mathrm{\Omega}^{2}}{4 m^{2}}, \ \forall j \in {N}.  
\end{equation}
Moreover, the following inequality can be derived by adding inequalities (\ref{eq03_07}) for all $j \in {N}$:
\begin{equation*} 
\|{\xi}\|^{2} \leq \sum_{j \in {N}}  l_{j}^{m}({\xi}_{j}) + n \frac{\mathrm{\Omega}^{2}}{4 m^{2}} \leq \mathrm{\Omega}^{2} + n \frac{\mathrm{\Omega}^{2}}{4 m^{2}} = \mathrm{\Omega}^2 \left( 1 + \frac{n}{4m^2} \right),
\end{equation*}
when $\xi$ satisfies $\sum_{j \in {N}} l_{j}^{m}({\xi}_{j}) \leq \mathrm{\Omega}^{2}$. Therefore, $\mathcal{U}_{{\gamma}_{m}}$ with ${\gamma}_{m} = \mathrm{\Omega} \sqrt{1 + {n} / {4m^{2}}}$ has $\mathcal{U}_{\gamma_{m}}  \supseteq \mathcal{U}_{\textrm{out}}^{m}$.  This implies that the value of $\underline{\gamma}_{m}$ is at most $\mathrm{\Omega} \sqrt{1 + {n}/{4m^2}}$. The second statement also holds, because $\sqrt{1 + {n}/{4m^2}} \to 1$ as $m \to \infty$. \qed 
\end{proof}

\section{Upper and lower bounds for the second-order cone-constrained binary knapsack problem} \label{sect04}

We now consider solutions of the robust knapsack problem with constraint (\ref{eq03_05}) and the problem with constraint (\ref{eq03_09}) which provide upper and lower bounds on the optimal value of \textbf{(SOCKP)}, respectively. 

The robust constraint (\ref{eq03_05}) with the inner approximation of the uncertainty set $\mathcal{U}$ can be represented as
\begin{equation} \nonumber
\sum_{j \in N} \hat{a}_{j} x_{j} + \beta^{m}_{\textrm{in}} (x) \leq b, 
\end{equation}
where
\begin{align} 
\beta^{m}_{\textrm{in}}(x) =  \textrm{max}_{}   & \ (\mathrm{\Sigma}^{1/2}x)' \xi \nonumber \\
\textrm{subject to} & \ \sum_{j \in N} u_{j}^{m}(\xi_{j}) \leq \mathrm{\Omega}^{2}, \nonumber \\
& \ \xi \in [0,\mathrm{\Omega}]^n. \nonumber 
\end{align}
Because $u_{j}^{m}$ is a piecewise linear function, we can represent not only $\xi_{j}$ but $u_{j}^{m}(\xi_{j})$ for $j \in {N}$ by introducing additional variables $z \in [0,1]^{m \times n}$ as
\begin{equation} \nonumber
\xi_{j}=\sum_{k \in {M}} (\pi_{j}^{k} - \pi_{j}^{k-1}) z_{j}^{k} = \mathrm{\Omega}  \sum_{k \in {M}} \frac{1}{m} z_{j}^{k}, \ \forall j \in {N}, 
\end{equation}
\begin{equation} \label{eq04_04} 
u_{j}^{m}(\xi_{j})=\sum_{k \in {M}} ((\pi_{j}^{k})^{2} - (\pi_{j}^{k-1})^{2}) z_{j}^{k} = \mathrm{\Omega}^2 \sum_{k \in {M}} \frac{2k-1}{m^2}   z_{j}^{k}, \ \forall j \in {N}. 
\end{equation}
Because we want to make $z_{j}^{k}>0$ only if $z_{j}^{1}=\cdots=z_{j}^{k-1}=1$, we introduce additional binary variables $v \in \{0,1\}^{(m+1) \times n}$. Then, we can reformulate the problem as
\begin{align} 
\beta^{m}_{\textrm{in}} (x) = \textnormal{max}  & \ \sum_{j \in {N}} \sum_{k \in {M}} d_{j}^{k} x_{j} z_{j}^{k} \nonumber \\
\textnormal{subject to} & \ \sum_{j \in {N}} \sum_{k \in {M}} f_{j}^{k} z_{j}^{k} \leq m^2, \label{eq04_05} \\
& \ v_{j}^{k} \leq z_{j}^{k} \leq v_{j}^{k-1}, \ \ \ \forall j \in {N}, k \in {M}, \label{eq04_06} \\
& \ v_{j}^{k} \in \{0,1\}, \ \ \ \forall j \in {N}, k \in {M} \cup \{0\}, \nonumber \\
& \ 0 \leq z_{j}^{k} \leq 1, \ \ \ \forall j \in {N}, k \in {M}, \nonumber
\end{align}
where $d_{j}^{k} = (\pi_{j}^{k} - \pi_{j}^{k-1}) \sigma_{j} =  \mathrm{\Omega} \cdot \sigma_{j} / m$, $f_{j}^{k} = 2k-1$ for $j \in {N}, k \in {M}$. Constraints (\ref{eq04_05}) are derived by dividing equations (\ref{eq04_04}) by ${\mathrm{\Omega}^2}/{m^2}$ and adding all of them. Constraints (\ref{eq04_06}) guarantee that $z_{j}^{k}>0$ only if $z_{j}^{1}=\cdots=z_{j}^{k-1}=1$.

Although the above problem is a mixed integer program, it can be solved as a linear program. When $d_{j}^{k} x_{j}$ has a positive value, the sequence of ratios $\{d_{j}^{k} x_{j}/f_{j}^{k} \ | \ k =1,2,\cdots,m\}$ is strictly decreasing, that is, $d_{j}^{1} x_{j}/f_{j}^{1} > \cdots >  d_{j}^{m} x_{j}/f_{j}^{m}$. When $x_{j}=0$, there is an optimal solution with $z_{j}^{k}=0$ for every $k \in {M}$. Therefore, there is an optimal solution $z$ such that $z_{j}^{k}$ is greater than zero only if $z_{j}^{1}, \cdots, z_{j}^{k-1}$ are equal to one for all $k \in {M}, j \in {N}$ even if constraints (\ref{eq04_06}) are relaxed. We can also eliminate all of the variables $v_{j}^{k}, j \in {N}, k \in {M} \cup \{0\}$ because these variables only appear in constraints (\ref{eq04_06}). Hence, the problem is equivalent to the following maximization problem:
\begin{align} 
\beta^{m}_{\textrm{in}}(x) = \textnormal{max}  & \ \sum_{j \in {N}} \sum_{k \in {M}}  d_{j}^{k} x_{j} z_{j}^{k} \nonumber \\
\textnormal{subject to} & \ \sum_{j \in {N}} \sum_{k \in M} f_{j}^{k} z_{j}^{k} \leq m^2, \nonumber \\
& \ 0 \leq z_{j}^{k} \leq 1, \ \ \ \forall k \in {M}, j \in {N}.  \nonumber
\end{align}
This problem is a bounded continuous knapsack problem, which is a linear program.

Similarly, the robust constraints (\ref{eq03_09}) with the outer approximation of the uncertainty set $\mathcal{U}$ can be represented as 
\begin{equation} \nonumber
\sum_{j \in N} \hat{a}_{j} x_{j} + \beta^{m}_{\textrm{out}} (x) \leq b, 
\end{equation}
where
\begin{align} 
\beta^{m}_{\textrm{out}}(x) = \textrm{max}_{}  & \ (\mathrm{\Sigma}^{1/2}x)' \xi \nonumber \\
\textrm{subject to} & \ \sum_{j \in N} l_{j}^{m}(\xi_{j}) \leq \mathrm{\Omega}^{2}, \label{eq04_08} \\
& \ \xi \in [0,\mathrm{\Omega}]^n. \nonumber 
\end{align}
Because $l_{j}^{m} (\xi_{j}) = u_{j}^{m} (\xi_{j}) - \mathrm{\Omega}^2 / 4m^{2}$, constraint (\ref{eq04_08}) can be written as
\begin{equation} \nonumber
\sum_{j \in N} u_{j}^{m}(\xi_{j}) \leq \mathrm{\Omega}^{2} + \sum_{j \in N} \frac{\mathrm{\Omega}^{2}}{4m^{2}} = \mathrm{\Omega}^{2} ( 1 + \frac{n}{4m^{2}} ).
\end{equation}
Then, by following the previous procedure to represent $\beta^{m}_{\textrm{in}} (x)$ again, we can obtain the following representation of $\beta^{m}_{\textrm{out}}$:
\begin{align} 
\beta^{m}_{\textrm{out}}(x) = \textnormal{max}  & \ \sum_{j \in {N}} \sum_{k \in {M}}  d_{j}^{k} x_{j} z_{j}^{k} \nonumber \\
\textnormal{subject to} & \ \sum_{j \in {N}} \sum_{k \in M} f_{j}^{k} z_{j}^{k} \leq m^2 + \frac{n}{4} , \nonumber \\
& \ 0 \leq z_{j}^{k} \leq 1, \ \ \ \forall k \in {M}, j \in {N}.  \nonumber
\end{align}

Because the two representations of $\beta^{m}_{\textrm{in}}(x)$ and $\beta^{m}_{\textrm{out}}(x)$ differ only in the right-hand side of the constraint, we can combine them as
\begin{equation} \label{eq04_10} 
\sum_{j \in N} \hat{a}_{j} x_{j} + \beta^{m}(x,\mathrm{\Delta}) \leq b,
\end{equation}
where
\begin{align*} 
\beta^{m}(x,\mathrm{\Delta}) = \textnormal{maximize}  & \ \sum_{j \in {N}} \sum_{k \in {M}} d_{j}^{k} x_{j} z_{j}^{k} \\
\textnormal{subject to} & \ \sum_{j \in {N}} \sum_{k \in {M}} f_{j}^{k} z_{j}^{k} \leq \mathrm{\Delta}, \\
& \ 0 \leq z_{j}^{k} \leq 1, \ \forall j \in {N}, k \in {M}.
\end{align*}
If $\mathrm{\Delta}=m^2$ or $\mathrm{\Delta}=m^2+n/4$, then $\beta^{m}(x,\mathrm{\Delta})$ is equal to $\beta^{m}_{\textrm{in}}(x)$ or $\beta^{m}_{\textrm{out}}(x)$, from which we can obtain the upper or lower bound on the optimal value of \textbf{(SOCKP)}, respectively.

After all, we need to solve the following robust BKP to obtain the upper and lower bounds for \textbf{(SOCKP)}.
\begin{align*} 
\textbf{(RKPm)} \ \textnormal{maximize} & \ \sum_{j \in {N}} p_{j} x_{j} \\
\textnormal{subject to} & \ \sum_{j \in {N}} \hat{a}_{j} x_{j} + \beta^{m}(x,\mathrm{\Delta}) \leq b, \\
& \ x \in \{0,1\}^{n}. 
\end{align*}

Because $\beta^{m}(x,\mathrm{\Delta})$ can be obtained by solving a linear program, we can solve \textbf{(RKPm)} as a mixed integer linear program by injecting the dual for $\beta^{m}(x,\mathrm{\Delta})$ into constraint (\ref{eq04_10}) as the approach attempted by Han et al. \cite{Han16}. This can be one of the polyhedral approximations of the solution set for second-order cone-constrained programs, which is different from the previous ones of Vielma et al. \cite{Vielma08}, Vielma et al. \cite{Vielma17} and Lubin et al. \cite{Lubin18}. However, we propose a combinatorial algorithm which can solve \textbf{(RKPm)} more efficiently. 

We simplify some notation for the later explanation. For this, we first arrange the order of $(j,k) \in {N} \times {M}$ in descending order of $d_{j}^{k} / f_{j}^{k}$ values. Ties are broken arbitrarily. Let ${L} := \{ 1, \cdots, nm \}$ be a set of the ordinal ranks of $(j,k)$, $j \in {N}$, $k \in {M}$. When $l$ is the ordinal rank of $(j,k)$, we let $j(l):=j$, $k(l):=k$, $d_{l} := d_{j}^{k}$, $f_{l} := f_{j}^{k}$, and $z_{l} := z_{j}^{k}$.

We also define some additional notation. Let ${L}_l$ for $l \in {L}$ be a subset of ${L}$ whose elements $l' \in {L}$ have greater ${d_{l'}}/{f_{l'}}$ value than the $l$-th element of ${L}$, that is, ${L}_l=\{l' \in {L} \ | \ {d_{l'}}/{f_{l'}} > {d_{l}}/{f_{l}} \}$,  $l \in {L}$. Let $L(S) := \{ l \in {L} \ | \ j(l) \in S\} \cup \{nm+1\}, S \subseteq {N}$, where the $(nm+1)$-th element is a dummy one with $d_{nm+1} = 0$ and $f_{nm+1} = \mathrm{\Delta}$. We also let $x(S)$ be the characteristic vector of $S \subseteq {N}$, that is, $x(S)_{j} = 1$ if $j \in S$, $x(S)_{j} = 0$ otherwise. 

For a given $x(S)$, $S \subseteq {N}$, we have $d_{l} x_{j(l)} = 0$ if $j(l) \notin S$, hence we can represent $\beta^{m}(x,\mathrm{\Delta})$ as
\begin{align*} 
\beta^{m}(x(S),\mathrm{\Delta}) = \textnormal{maximize}  & \ \sum_{l \in L(S)} d_{l} z_{l} \\
\textnormal{subject to} & \ \sum_{l \in L(S)} f_{l} z_{l} \leq \mathrm{\Delta}, \\
& \ 0 \leq z_{l} \leq 1, \ \forall l \in L(S),
\end{align*}
for $S \subseteq {N}$, $\mathrm{\Delta} \geq m^{2}$. 

Let $l^{*}$ be $l \in L(S)$, $S \subseteq {N}$ satisfying $\sum_{l \in L(S) | l < l^{*}} f_{l} < \mathrm{\Delta}$ and $\sum_{l \in L(S) | l < l^{*}} f_{l} + f_{l^{*}} \geq \mathrm{\Delta}$, and let $T^{*}$ be a set  $\{ l \in L(S) \ | \ l < l^{*}\}$. Then, it has an optimal solution $z^{*}$ satisfying $z^{*}_{l}=1$ for $l \in T^{*}$, $z^{*}_{l}=0$ for $l \notin T^{*} \cup \{l^{*}\}$, and $z^{*}_{l^{*}} = (\mathrm{\Delta}-\sum_{l \in T^{*}} f_{l}) / f_{l^{*}}$. We note that $z^{*}$ can be obtained in $\mathrm{O}(|{L}| \log |{L}|) = \mathrm{O}(nm \log (nm))$ time complexity for sorting the values of $d_{l} / f_{l}$, $l \in L(S)$. 

Now, we can show that an optimal solution to \textbf{(RKPm)} can be obtained by solving the ordinary BKP repeatedly with the same objective function $\sum_{j \in {N}} p_{j} x_{j}$, but with different constraints. 

\begin{theorem} \label{thm04_01}
Let ${B}$ be the set of feasible solutions to \textbf{(RKPm)}:
\begin{equation} \nonumber
{B} = \left\{ x \in \{0,1\}^{n} \ | \ \sum_{j \in {N}} \hat{a}_{j} x_{j} + \beta^{m}(x,\mathrm{\Delta}) \leq b \right\}.
\end{equation}
Let ${B}_{l}$, $l \in {L} \cup \{nm+1\}$ be the set of feasible solutions of a BKP represented by
\begin{equation} \nonumber
{B}_{l} = \left\{ x \in \{0,1\}^{n} \ | \ \sum_{j \in {N}} \hat{a}_{j} x_{j} + \sum_{l' \in {L}_{l}} (d_{l'}-\frac{d_{l}}{f_{l}} f_{l'}) x_{j(l')} \leq b - \frac{d_{l}}{f_{l}} \mathrm{\Delta} \right\},
\end{equation}
for $l \in {L}$ and 
\begin{equation} \nonumber
{B}_{nm+1} = \left\{ x \in \{0,1\}^{n} \ | \ \sum_{j \in {N}} \hat{a}_{j} x_{j} + \sum_{l \in {L}} d_{l} x_{j(l)} \leq b \right\}. 
\end{equation}
Then, 
\begin{equation} \nonumber
{B} = \bigcup_{l \in \{ \hat{l}, \hat{l}+1,\cdots,nm-1,nm+1\}} {B}_{l},
\end{equation}
where 
$\hat{l} = \textnormal{min} \{ l \in {L} \ | \ \sum_{l'=1}^{l} f_{l'} \geq \mathrm{\Delta} \}$.
\end{theorem}
\begin{proof} 
Let $\mathcal{T}(S)$ be the set of all $(T,t)$ for a set $S \subseteq {N}$, where $T$ is a proper subset of $L(S)$ and $t \notin T$ is one element in $L(S)$ such that $\sum_{l \in T} f_{l} < \mathrm{\Delta}$, $\sum_{l \in T} f_{l} + f_{t} \geq \mathrm{\Delta}$. We note that $\mathcal{T}(S)$ is nonempty because of the $(nm+1)$-th dummy element.

We first show that ${B} \subseteq \bigcup_{l \in \{ \hat{l}, \hat{l}+1,\cdots,nm-1,nm+1\}} B_{l}$. Suppose that the characteristic vector $x(S)$ of $S \subseteq {N}$ is in ${B}$. Because it is clear that $(T^{*},l^{*}) \in \mathcal{T}(S)$, we have
\begin{align*} 
& \ \sum_{j \in S} \hat{a}_{j} + \beta^{m}(x(S),\mathrm{\Delta}) \\
= & \ \sum_{j \in S} \hat{a}_{j} + \max_{(T,t) \in \mathcal{T}(S)} \left\{ \sum_{l \in T} d_{l} + \left( \frac{\mathrm{\Delta} - \sum_{l \in T} f_{l}}{f_{t}} \right) d_{t} \right\} \\
= & \ \sum_{j \in S} \hat{a}_{j} + \sum_{l \in T^{*}} d_{l} + \left( \frac{\mathrm{\Delta} - \sum_{l \in T^{*}} f_{l} }{f_{l^{*}}} \right) d_{l^{*}} \ \leq b.
\end{align*}
$T^{*} \cap L_{l^{*}} = L(S) \cap L_{l^{*}}$ is derived from $L(S) \cap L_{l^{*}} \subseteq T^{*} \subseteq L(S)$, and $T^{*} \setminus L_{l^{*}} = \{ l \in T^{*} \ | \ d_{l} / f_{l} = d_{l^{*}} / f_{l^{*}} \}$. Then, we can express the following equation:
\begin{align*} 
  & \ \sum_{j \in S} \hat{a}_{j} + \sum_{l \in T^{*}} d_{l} + \left( \frac{\mathrm{\Delta} - \sum_{l \in T^{*}} f_{l} }{f_{l^{*}}} \right) d_{l^{*}} \\
= & \ \sum_{j \in S} \hat{a}_{j} + \sum_{l \in T^{*} \cap L_{l^{*}}} d_{l} + \sum_{l \in T^{*} \setminus L_{l^{*}}} d_{l} + \left( \frac{\mathrm{\Delta} - \sum_{l \in T^{*}} f_{l} }{f_{l^{*}}} \right) d_{l^{*}} \\
= & \ \sum_{j \in S} \hat{a}_{j} + \sum_{l \in T^{*} \cap L_{l^{*}}} d_{l} + \sum_{l \in T^{*} \setminus L_{l^{*}}} \frac{d_{l^{*}}}{f_{l^{*}}} f_{l} + \left( \frac{\mathrm{\Delta} - \sum_{l \in T^{*}} f_{l} }{f_{l^{*}}} \right) d_{l^{*}} \\
= & \ \sum_{j \in S} \hat{a}_{j} + \sum_{l \in T^{*} \cap L_{l^{*}}} d_{l} + \left( \frac{\mathrm{\Delta} - \sum_{l \in T^{*}} f_{l} + \sum_{l \in T^{*} \setminus L_{l^{*}}} f_{l}}{f_{l^{*}}} \right) d_{l^{*}} \\
= & \ \sum_{j \in S} \hat{a}_{j} + \sum_{l \in T^{*} \cap L_{l^{*}}} d_{l} + \left( \frac{\mathrm{\Delta} - \sum_{l \in T^{*} \cap L_{l^{*}}} f_{l}}{f_{l^{*}}} \right) d_{l^{*}} \\
= & \ \sum_{j \in S} \hat{a}_{j} + \sum_{l \in T^{*} \cap L_{l^{*}}} \left( d_{l} - \frac{d_{l^{*}}}{f_{l^{*}}} f_{l} \right) + \frac{d_{l^{*}}}{f_{l^{*}}} \mathrm{\Delta} \\
= & \ \sum_{j \in S} \hat{a}_{j} + \sum_{l \in L(S) \cap L_{l^{*}}} \left( d_{l} - \frac{d_{l^{*}}}{f_{l^{*}}} f_{l} \right) + \frac{d_{l^{*}}}{f_{l^{*}}} \mathrm{\Delta} \leq b
\end{align*}
Therefore, 
\begin{equation*} 
\sum_{j \in S} \hat{a}_{j} + \sum_{l \in L(S) \cap L_{l^{*}}} \left( d_{l} - \frac{d_{l^{*}}}{f_{l^{*}}} f_{l} \right) \leq b - \frac{d_{l^{*}}}{f_{l^{*}}} \mathrm{\Delta},
\end{equation*}
and then $x(S) \in B_{l^{*}}$.

Next, we show that $\bigcup_{l \in \{ \hat{l}, \hat{l}+1,\cdots,nm-1,nm+1\}} B_{l} \subseteq {B}$. Suppose that the characteristic vector $x(S)$ of a nonempty $S \subseteq {N}$ is in $B_{l'}$ for some $l' \in \{ \hat{l}, \hat{l}+1,\cdots,nm-1,nm+1\}$. It is trivial that $x(S) \in {B}$ if $x(S) \in B_{nm+1}$, so we assume $l' \leq nm-1$. If a solution $x(S)$ is in $B_{l'}$, it implies that 
\begin{equation*} 
\sum_{j \in S} \hat{a}_{j} + \sum_{l \in L(S) \cap L_{l'}} \left( d_{l} - \frac{d_{l'}}{f_{l'}} f_{l} \right) \leq b - \frac{d_{l'}}{f_{l'}} \mathrm{\Delta}.
\end{equation*}
We also define $l^{*}$ and $T^{*}$ as before. Then,
\begin{equation*} 
\sum_{j \in S} \hat{a}_{j} + \beta^{m}(x(S),\mathrm{\Delta}) =\ \sum_{j \in S} \hat{a}_{j} + \sum_{l \in T^{*}} d_{l} + \left( \frac{\mathrm{\Delta} - \sum_{l \in T^{*}} f_{l} }{f_{l^{*}}} \right) d_{l^{*}}.
\end{equation*}
If $d_{l^{*}} / f_{l^{*}} \leq d_{l'} / f_{l'}$,
\begin{align*} 
  & \ \sum_{j \in S} \hat{a}_{j} + \sum_{l \in T^{*}} d_{l} + \left( \frac{\mathrm{\Delta} - \sum_{l \in T^{*}} f_{l} }{f_{l^{*}}} \right) d_{l^{*}} \\ 
 = & \ \sum_{j \in S} \hat{a}_{j} + \sum_{l \in T^{*} \cap L_{l'}} d_{l} + \sum_{l \in T^{*} \setminus L_{l'}} d_{l} + \frac{d_{l^{*}}}{f_{l^{*}}}\left(\mathrm{\Delta} - \sum_{l \in T^{*}} f_{l}\right)  \\
\leq & \ \sum_{j \in S} \hat{a}_{j} + \sum_{l \in T^{*} \cap L_{l'}} d_{l} + \sum_{l \in T^{*} \setminus L_{l'}} \frac{d_{l'}}{f_{l'}} f_{l} + \frac{d_{l^{'}}}{f_{l^{'}}}\left(\mathrm{\Delta} - \sum_{l \in T^{*}} f_{l}\right) \ \ (\textrm{since} \ l \notin L_{l'}.) \\
= & \ \sum_{j \in S} \hat{a}_{j} + \sum_{l \in T^{*} \cap L_{l'}} \left( d_{l} - \frac{d_{l'}}{f_{l'}} f_{l} \right)+ \sum_{l \in T^{*}} \frac{d_{l'}}{f_{l'}} f_{l} + \frac{d_{l^{'}}}{f_{l^{'}}}\left(\mathrm{\Delta} - \sum_{l \in T^{*}} f_{l}\right)   \\
= & \ \sum_{j \in S} \hat{a}_{j} + \sum_{l \in T^{*} \cap L_{l'}} \left( d_{l} - \frac{d_{l'}}{f_{l'}} f_{l} \right) +  \frac{d_{l^{'}}}{f_{l^{'}}} \mathrm{\Delta}  \\
\leq & \ \sum_{j \in S} \hat{a}_{j} + \sum_{l \in L(S) \cap L_{l'}} \left( d_{l} - \frac{d_{l'}}{f_{l'}} f_{l} \right) +  \frac{d_{l^{'}}}{f_{l^{'}}} \mathrm{\Delta}.
\end{align*}
If $d_{l^{*}} / f_{l^{*}} > d_{l'} / f_{l'}$,
\begin{align*} 
  & \ \sum_{j \in S} \hat{a}_{j} + \sum_{l \in T^{*}} d_{l} + \left( \frac{\mathrm{\Delta} - \sum_{l \in T^{*}} f_{l} }{f_{l^{*}}} \right) d_{l^{*}} \\ 
= & \ \sum_{j \in S} \hat{a}_{j} + \sum_{l \in T^{*}} \left( d_{l} - \frac{d_{l'}}{f_{l'}} f_{l} \right) + \left( \frac{\mathrm{\Delta} - \sum_{l \in T^{*}} f_{l} }{f_{l^{*}}} \right) \left( d_{l^{*}} - \frac{d_{l'}}{f_{l'}} f_{l^{*}} \right) \\
 & \ \ \ \ \ \ \ + \sum_{l \in T^{*}} \frac{d_{l'}}{f_{l'}} f_{l} + \left( \frac{\mathrm{\Delta} - \sum_{l \in T^{*}} f_{l} }{f_{l^{*}}} \right) \frac{d_{l'}}{f_{l'}} f_{l^{*}} \\
\leq & \ \sum_{j \in S} \hat{a}_{j} + \sum_{l \in T^{*} \cup \{l^{*} \}} \left( d_{l} - \frac{d_{l'}}{f_{l'}} f_{l} \right) + \sum_{l \in T^{*}} \frac{d_{l'}}{f_{l'}} f_{l}  + \frac{d_{l'}}{f_{l'}} \left( \mathrm{\Delta} - \sum_{l \in T^{*}} f_{l} \right) \\
= & \ \sum_{j \in S} \hat{a}_{j} + \sum_{l \in T^{*} \cup \{l^{*} \}} \left( d_{l} - \frac{d_{l'}}{f_{l'}} f_{l} \right) + \frac{d_{l'}}{f_{l'}} \mathrm{\Delta} \\
\leq & \ \sum_{j \in S} \hat{a}_{j} + \sum_{l \in L(S) \cap L_{l'}} \left( d_{l} - \frac{d_{l'}}{f_{l'}} f_{l} \right) +  \frac{d_{l^{'}}}{f_{l^{'}}} \mathrm{\Delta}. \ \ (\textrm{since} \  T^{*} \cup \{l^{*} \} \subseteq L(S) \cap L_{l'}.) 
\end{align*}
Therefore,
\begin{equation*} 
\sum_{j \in S} \hat{a}_{j} + \beta^{m}(x(S),\mathrm{\Delta}) \leq \sum_{j \in S} \hat{a}_{j} + \sum_{l \in L(S) \cap L_{l'}} \left( d_{l} - \frac{d_{l'}}{f_{l'}} f_{l} \right) +  \frac{d_{l^{'}}}{f_{l^{'}}} \mathrm{\Delta},
\end{equation*}
and then,
\begin{equation*} 
\sum_{j \in S} \hat{a}_{j} + \beta^{m}(x(S),\mathrm{\Delta}) \leq b.
\end{equation*}
This shows $x(S) \in {B}$. Therefore, ${B} = \bigcup_{l \in \{ \hat{l}, \hat{l}+1,\cdots,nm-1,nm+1\}} {B}_{l}$. \qed
\end{proof}

By Theorem \ref{thm04_01}, to solve \textbf{(RKPm)}, we only need to solve ordinary BKPs with solution set $B_{l}$, $l = \hat{l},\cdots,nm-1,nm+1$ and take the best value. We obtain an upper bound when $\mathrm{\Delta}=m^2$, and we obtain a lower bound when $\mathrm{\Delta}=m^2+n/4$. The solution obtained when $\mathrm{\Delta}=m^2+n/4$ is guaranteed to be a good feasible solution to \textbf{(SOCKP)}.


\section{Probability guarantees of feasible solutions to \textbf{(RKPm)}} \label{sect05}

In this section, we consider the bounds on the probability that feasible solutions to \textbf{(RKPm)} satisfy the chance constraint (\ref{eq01_01}) or the distributionally robust chance constraint (\ref{eq01_02}). We consider the case of $\mathrm{\Delta} = m^2$ in \textbf{(RKPm)}. That is we use the inner polyhedral approximation $\mathcal{U}^{m}_{in}$ of the uncertainty set $\mathcal{U}$. 

An optimal solution to \textbf{(RKPm)} with $\mathrm{\Delta} = m^2$ provides an upper bound on the optimal value of \textbf{(CKP)}, but the optimal solution may not satisfy constraint (\ref{eq01_01}). Han et al. \cite{Han16} investigated the lower bound on the probability that feasible solutions to the robust knapsack problem obtained with their own inner polyhedral approximation of $\mathcal{U}$ satisfy constraint (\ref{eq01_01}). We can strengthen the bounds on the probability by using our approximation scheme.    

\begin{proposition} \label{prop05_01}
Assume that $a_{j}$, $j \in {N}$ are independent normally distributed random variables. Let $\underline{B}^{m}_{\textrm{CKP}}$ be the set of feasible solutions to \textbf{(RKPm)} with $\mathrm{\Omega} = \mathrm{\Phi}^{-1} (\rho)$ and $\mathrm{\Delta} = m^2$. Let $\underline{\rho}^{m}_{\textrm{CKP}} := \min_{x \in \underline{B}^{m}_{\textrm{CKP}}} \textnormal{Prob} \left\{ a' x \leq b \right\}$. Then, 
\begin{equation} \label{eq05_01}
\rho - \underline{\rho}^{m}_{\textrm{CKP}} \leq \frac{1}{\sqrt{2\pi}} \mathrm{\Phi}^{-1} (\rho) \left( 1 - \sqrt{1 - \frac{n}{4m^{2}}} \right) \mathrm{exp} \left( - \frac{1}{2} (\mathrm{\Phi}^{-1}\left(\rho) \right)^{2}  \left( 1 - \frac{n}{4m^{2}} \right) \right),
\end{equation}
for ${m} \geq {\sqrt{n}/2}$. Moreover,
\begin{equation*}  
\lim_{m \to \infty} \underline{\rho}^{m}_{\textrm{CKP}} \geq \rho.
\end{equation*}
\end{proposition}
\begin{proof}
We refer to Proposition 2 in Han et al. \cite{Han16} for details of the proof. A major difference is that $\bar{\gamma}_{m} \geq \mathrm{\Omega} \sqrt{1 - {n}/{4m^2}}$ in Proposition \ref{prop03_01} of this paper can be used instead of  ${r}_{m} \geq \mathrm{\Phi}^{-1}(\rho) \sqrt{1 - {n}/{4m}}$ in Han et al. \cite{Han16} when $\mathrm{\Omega} = \mathrm{\Phi}^{-1}(\rho)$. This improvement on bound leads to the improved result (\ref{eq05_01}).
\qed
\end{proof}

From the improvement of Proposition \ref{prop05_01}, we can also decrease the minimum number of the linear segments $m$ required to make feasible solutions to \textbf{(RKPm)} satisfy constraint (\ref{eq01_01}) with the guaranteed probability as shown in the following corollary.
 
\begin{corollary} \label{cor05_01}
For fixed $\rho<1$, the minimum number of $m$ that makes $\rho - \underline{\rho}^{m}_{\textrm{CKP}}$ less than $\delta > 0$ is $\mathrm{O} ( \sqrt{{n}/{\delta}} )$.
\end{corollary}
\begin{proof}
The proof is parallel to the proof of Proposition 3 in Han et al. \cite{Han16} with $\sqrt{1 - n/4m}$ replaced by $\sqrt{1 - n/4m^{2}}$. \qed  
\end{proof}

\noindent For example, Han et al. \cite{Han16} observed that the difference between $\underline{\rho}^{m}_{\textrm{CKP}}$ and $\rho$ is less than $0.01$ when $\rho=0.95$ and $m \geq 3n$. Our approximation scheme improves the results so that the difference is less than $0.01$ when $\rho=0.95$ and $m \geq \sqrt{3n} \simeq 1.74\sqrt{n}$ under the same assumptions.

Now, we consider the cases of distributionally robust BKPs. When the ambiguity set $\mathcal{D}_{a}$ is $\mathcal{D}_{a}^{1}$ with the first and second moment information available in \textbf{(DCKP)}, we can derive similar results on the probability bounds for constraint (\ref{eq01_02}).

\begin{proposition} \label{prop05_02}
Assume that $a$ is arbitrary random vector of independent random variables $a_{j}$, $j \in {N}$ in $\mathcal{D}_{a}^{1}$. Let $\underline{B}^{m}_{\textrm{DCKP}}$ be the set of feasible solutions to \textbf{(RKPm)} with $\mathrm{\Omega} = \sqrt{\rho/(1-\rho)}$ and $\mathrm{\Delta} = m^2$. Let $\underline{\rho}^{m}_{\textrm{DCKP}} := \min_{x \in \underline{B}^{m}_{\textrm{DCKP}}} \left\{ \inf_{a \in \mathcal{D}_{a}^{1}}  \textnormal{Prob} \left\{ a' x \leq b \right\} \right\}$.
Then, 
\begin{equation*} 
\rho - \underline{\rho}^{m}_{\textrm{DCKP}} \leq \frac{\rho(1-\rho)}{{4m^2}/{n} - \rho}, 
\end{equation*}
for ${m} \geq {\sqrt{n}/2}$. Moreover,
\begin{equation*}  
\lim_{m \to \infty} \underline{\rho}^{m}_{\textrm{DCKP}} \geq \rho.
\end{equation*}
\end{proposition}
\begin{proof}
By Proposition \ref{prop03_01}, an ellipsoidal uncertainty set $\mathcal{U}_{\mathrm{\Omega} \sqrt{1-{n}/{4m^2}}}$ is a subset of the uncertainty set $\mathcal{U}_{\textrm{in}}^{m}$, where $\mathrm{\Omega}=\sqrt{\rho/(1-\rho)}$. Then, $x \in \underline{B}^{m}_{\textrm{DCKP}}$ satisfies $\sum_{j \in {N}} a_{j} x_{j} \leq b$ for every $ a \in \mathcal{U}_{\mathrm{\Omega} \sqrt{1-{n}/{4m^2}}}$. This is equivalent to the following chance constraint: 
\begin{equation*} 
\inf_{a \in \mathcal{D}_{a}^{1}} \textnormal{Prob} \left\{ a' x \leq b \right\} \geq q^{-1} \left( \mathrm{\Omega} \sqrt{1-\frac{n}{4m^2}} \right),
\end{equation*}
where $\omega^{-1}(\gamma) = 1-1/(1+\gamma^2)$, $\gamma \geq 0$ is an inverse function of $\omega(\rho)=\sqrt{\rho/(1-\rho)} = \mathrm{\Omega}$, $0 \leq \rho < 1$. Therefore, $\underline{\rho}^{m}_{\textrm{DCKP}} \geq \omega^{-1} ( \mathrm{\Omega}  \sqrt{1-{n}/{4m^2}} )$, and then
\begin{align*} 
\rho - \underline{\rho}^{m}_{\textrm{DCKP}} & \ \leq \ \rho - \omega^{-1} \left( \mathrm{\Omega} \sqrt{1-\frac{n}{4m^2}} \right) \\
& \ = \ \rho - \left( 1 - \frac{1}{1+\left( \sqrt{\rho/(1-\rho)} \sqrt{1-{n}/{4m^2}} \right)^2} \right) \\
& \ = \ \frac{1-\rho}{(1-\rho)+\rho(1-n/4m^2)} - (1-\rho) \\
& \ = \ \frac{1-(1-\rho(n/4m^2))}{1-\rho(n/4m^2)} (1-\rho) \\
& \ = \  \frac{\rho(1-\rho)}{4m^{2}/n-\rho}. \qed  
\end{align*} 
\end{proof}

We also obtain the minimum number of $m$ required to make feasible solutions to \textbf{(RKPm)} satisfy constraint (\ref{eq01_01}) with the guaranteed probability by the following corollary.

\begin{corollary} \label{cor05_02}
For fixed $\rho<1$, the minimum number of $m$ that makes $\rho - \underline{\rho}^{m}_{\textrm{DCKP}}$ less than $\delta > 0$ is expressed as
\begin{equation*}  
\underline{m} := \left\lceil \sqrt{ \left( \frac{\rho(1-\rho)}{4\delta} + \frac{\rho}{4} \right) n  } \ \right\rceil. 
\end{equation*}
Moreover, $\underline{m} = \mathrm{O} ( \sqrt{{n}/{\delta}} )$.
\end{corollary}

\begin{proof}
It is derived from 
\begin{equation*} 
\rho - \underline{\rho}^{m}_{\textrm{DCKP}} \leq \frac{\rho(1-\rho)}{4m^{2}/n-\rho} \leq \delta. \\
\end{equation*}
\qed
\end{proof}

\section{An exact algorithm for the second-order cone-constrained binary knapsack problem} \label{sect06}

Propositions \ref{prop05_01} and \ref{prop05_02} in section \ref{sect05} show that the value of the minimum probability satisfying the chance constraint for feasible solutions to \textbf{(RKPm)} with $\mathrm{\Delta} = m^{2}$ converges to a value greater than or equal to $\rho$ as $m \to \infty$. However, it does not guarantee that the optimal solution to \textbf{(RKPm)} with $\mathrm{\Delta} = m^{2}$ is feasible to \textbf{(SOCKP)} at some finite value of $m$. The following proposition shows the finiteness of $m$, which makes the optimal solution of \textbf{(RKPm)} with $\mathrm{\Delta}=m^{2}$ become feasible to \textbf{(SOCKP)}; thus, it is optimal to \textbf{(SOCKP)}.

\begin{proposition} \label{prop06_00}
An optimal solution to \textbf{(RKPm)} with $\mathrm{\Delta}=m^{2}$ is also an optimal solution to \textbf{(SOCKP)} if the value of $m$ is sufficiently large.  
\end{proposition}
\begin{proof}
Let $x^{*} \in  \{0,1\}^{n}$ be an infeasible solution to \textbf{(SOCKP)}. Then, $\mathrm{\Omega} > (b-\hat{a}'x^{*}) / ||\mathrm{\Sigma}^{1/2} x^{*} ||$ and $x^{*}$ is also infeasible to the robust BKP with uncertainty set $\mathcal{U}_{\mathrm{\Omega}}$. We then choose a value $\mathrm{\Omega}' = \mathrm{\Omega} \sqrt{1-n/4m^{'2}}$ with an integer $m'$ value and satisfying $\mathrm{\Omega} > \mathrm{\Omega}' > (b-\hat{a}'x^{*}) / ||\mathrm{\Sigma}^{1/2} x^{*} ||$. Then, $x^{*}$ is also infeasible to the robust BKP with uncertainty set $\mathcal{U}_{\mathrm{\Omega'}}$, which is equivalent to \textbf{(SOCKP)} whose coefficient $\mathrm{\Omega}$ is replaced by $\mathrm{\Omega}'$.

By Proposition \ref{prop03_01}, $m = m'$ makes $x^{*}$ infeasible to \textbf{(RKPm)} with $\mathrm{\Delta}=m^{2}$ because $\{ x \in  \{0,1\}^{n} \ | \ \sum_{j \in {N}} a_{j} x_{j} \leq b, \forall a \in \mathcal{U}_{in}^{m'} \}$ is a safe approximation of $\{ x \in \{0,1\}^{n} \ | \  \sum_{j \in {N}} a_{j} x_{j} \leq b, \forall a \in \mathcal{U}_{\mathrm{\Omega}'} \}$ i.e. $\mathcal{U}_{\mathrm{\Omega}'} \subseteq \mathcal{U}_{in}^{m'} \subseteq \mathcal{U}_{\mathrm{\Omega}}$. Therefore, every infeasible solution of \textbf{(SOCKP)} is also infeasible to \textbf{(RKPm)} with $\mathrm{\Delta}=m^{2}$ when the value of $m$ is greater than or equal to $m'$. This shows that an optimal solution to \textbf{(RKPm)} with $\mathrm{\Delta}=m^{2}$ is a feasible solution; thus it is an optimal solution to \textbf{(SOCKP)} when the value of $m$ is sufficiently large.  \qed 
\end{proof}

Tentatively, we assume that $\hat{a}_{j}$ and $\sigma_{j}$, $j \in {N}$ and $\mathrm{\Omega}$ are also integer coefficients as $p_{j}$, $j \in {N}$ and $b$. Then, we can obtain the value of $m$, which makes the optimal solution to \textbf{(RKPm)} with $\mathrm{\Delta} = m^{2}$ become optimal to \textbf{(SOCKP)}.

\begin{proposition} \label{prop06_01}
We assume that every coefficient for \textbf{(SOCKP)} is integer-valued. An optimal solution to \textbf{(RKPm)} with $\mathrm{\Delta}=m^{2}$ is also an optimal solution to \textbf{(SOCKP)} if the value of $m$ is at least $m^{*} := \left\lceil \mathrm{\Omega}/2 \cdot \sqrt{n\sum_{j \in {N}} \sigma_{j}^{2}} \right\rceil+1$.   
\end{proposition}
\begin{proof}
Let $x^{*} \in  \{0,1\}^{n}$ be an infeasible solution to \textbf{(SOCKP)} which has the greatest positive value of $(b-\hat{a}'x^{*}) / ||\mathrm{\Sigma}^{1/2} x^{*} || < \mathrm{\Omega}$. For $m \geq m^*$, the following inequalities are satisfied: $\mathrm{\Omega}^{2} \cdot n/4m^2 <  1/(\sum_{j \in {N} } \sigma_{j}^{2}) \leq   1/||\mathrm{\Sigma}^{1/2}  x^{*} ||^2 =  1/(\sum_{j \in {N} } \sigma_{j}^{2} (x_{j}^{*})^{2})$, where the first inequality is derived from $m / \sqrt{n} > \mathrm{\Omega}/2 \cdot \sqrt{ \sum_{j \in {N}} \sigma_{j}^{2} }$. 

Because $\mathrm{\Omega}^{2} - (b-\hat{a}'x^{*})^{2}  / ||\mathrm{\Sigma}^{1/2} x^{*} ||^{2}$ has a positive value and all coefficients are integer, $\mathrm{\Omega}^{2} - (b-\hat{a}'x^{*})^{2} / ||\mathrm{\Sigma}^{1/2} x^{*} ||^{2}$ is at least $1 / ||\mathrm{\Sigma}^{1/2} x^{*} ||^{2}$. Therefore, $\mathrm{\Omega}^{2} (1-n/4m^2) > \mathrm{\Omega}^{2} - 1/||\mathrm{\Sigma}^{1/2}  x^{*} ||^2 \geq (b-\hat{a}'x^{*})^{2}/||\mathrm{\Sigma}^{1/2} x^{*} ||^{2}$, and then $\mathrm{\Omega} \sqrt{1-n/4m^2}  \cdot ||\mathrm{\Sigma}^{1/2} x^{*} || > b-\hat{a}'x^{*}$. 

We can see that $x^{*}$ is also infeasible solution to \textbf{(SOCKP)} whose coefficient $\mathrm{\Omega}$ is replaced by $\mathrm{\Omega}^{*} := \mathrm{\Omega} \sqrt{1-n/4m^2}$, which is equivalent to the robust BKP with uncertainty set $\mathcal{U}_{\mathrm{\Omega}^{*}}$. By Proposition \ref{prop03_01}, $x^{*}$ is infeasible to \textbf{(RKPm)} when $m \geq m^{*}$ because $\mathcal{U}_{\mathrm{\Omega}^{*}} \subseteq \mathcal{U}_{in}^{m^{*}} \subseteq \mathcal{U}_{\mathrm{\Omega}}$. Therefore, a feasible solution to \textbf{(RKPm)} with $\mathrm{\Delta}=(m^{*})^{2}$ is feasible to \textbf{(SOCKP)}. Because \textbf{(RKPm)} with $\mathrm{\Delta}=(m^{*})^{2}$ and \textbf{(SOCKP)} have the same objective function, an optimal solution to \textbf{(RKPm)} with $\mathrm{\Delta}=(m^{*})^{2}$ is optimal to \textbf{(SOCKP)}. \qed
\end{proof}

From Proposition \ref{prop06_01}, a pseudo-polynomial time algorithm for \textbf{(SOCKP)} can be derived. 

\begin{theorem} \label{thm06_01}
There is a pseudo-polynomial time algorithm to obtain an exact optimal solution to \textbf{(SOCKP)} when the algorithm uses a pseudo-polynomial time algorithm as a subroutine for the ordinary BKP.   
\end{theorem}                

DP-profit \cite{Kellerer04} is a pseudo-polynomial time algorithm for the ordinary BKPs whose coefficients of weights and capacity are non-integer, while objective coefficients are integer-valued. It can be used to solve the ordinary BKPs whose solution set is $B_{l}$ for each $l \in {L}$. Therefore, Theorem \ref{thm06_01} shows the existence of a pseudo-polynomial time algorithm for \textbf{(SOCKP)}. The time complexity of DP-profit is $\mathrm{O}(nU)$, where $U$ is an upper bound on the optimal value of the ordinary BKP. Therefore, the computational time required to evaluate the exact solution to \textbf{(SOCKP)} is at most $\mathrm{O}(m^{*}n^{2}U) = \mathrm{O}\left( \mathrm{\Omega} \sqrt{\sum_{j \in {N}} \sigma_{j}^{2}} n^{2.5} U \right)$, assuming the integrality of coefficients $\hat{a}_{j}$ and $\sigma_{j}$, $j \in N$ and $\Omega$.

In practical applications, coefficients can have non-integer values. In particular, when parameter $\mathrm{\Omega}$ is equal to $\mathrm{\Phi}^{-1}(\rho)$, $0.5 \leq \rho < 1$ or $\sqrt{\rho/(1-\rho)}$, $0 \leq \rho<1$ and the means and standard deviations are obtained by sampling, these coefficients may not be integers. However, we can extend the previous results to the case where these coefficients are finite decimals without loss of generality. The coefficients are scaled to integers without affecting the solution set of \textbf{(SOCKP)} by multiplying both sides of the constraint by $10^{s}$, where $s$ is the minimum value that  makes all $10^{s}\hat{a}_{1},\cdots,10^{s}\hat{a}_{n}$ and $10^{s}\mathrm{\Omega}\hat{\sigma}_{1},\cdots,10^{s}\mathrm{\Omega}\hat{\sigma}_{n}$ be integers. We note that $s$ is less than or equal to the maximum length of the fractional parts of $\hat{a}_{1}, \cdots, \hat{a}_{n}, \mathrm{\Omega}\hat{\sigma}_{1}, \cdots, \mathrm{\Omega}\hat{\sigma}_{n}$. Therefore, Theorem \ref{thm06_01} can also be applied when the coefficients of the problem are finite decimals.

Although the value of $m^{*}$ is finite, it can be very large in practice. It is possible that an optimal solution to \textbf{(SOCKP)} can be obtained for a value of $m$ that is much smaller than $m^{*}$. Hence, instead of using $m^{*}$ directly, we can start with a small value of $m$ and increase it iteratively until an optimal solution to \textbf{(RKPm)} with $\mathrm{\Delta}=m^{2}$ becomes feasible to \textbf{(SOCKP)}. This approach is summarized as Algorithm \ref{algo06_01} for \textbf{(SOCKP)}:

\begin{algorithm} [H]
\caption{Exact algorithm for \textbf{(SOCKP)}}
\label{algo06_01}
\begin{algorithmic}[1]
\STATE $\hat{m}=\sqrt{n}/2$, $\textrm{itr}=0$, $x^{*}=\{1,\cdots,1\}$, ${z}^{*}=c'x^{*}$.
\WHILE{$\hat{a}'x^{*} + \mathrm{\Omega} ||\mathrm{\Sigma}^{1/2} x^{*} || > b$}
\STATE $\hat{m} \leftarrow 2 \times \hat{m}$, $\textrm{itr} \leftarrow \textrm{itr}+1$. 
\STATE Solve \textbf{(RKPm)} with $m = \hat{m}$.
\STATE $x^{*} \leftarrow \textrm{optimal solution}$, ${z}^{*} \leftarrow \textrm{optimal value}$.
\ENDWHILE
\RETURN Exact optimal solution $x^{*}$ to \textbf{(SOCKP)} and its optimal value ${z}^{*}$.
\end{algorithmic}
\end{algorithm}

By Proposition \ref{prop06_01}, Algorithm \ref{algo06_01} is guaranteed to stop before $\hat{m}$ exceeds $m^{*}$ value. We can also see that Algorithm \ref{algo06_01} is a pseudo-polynomial time algorithm when the algorithm uses a pseudo-polynomial time algorithm for the ordinary BKPs.  Algorithm \ref{algo06_01} can be terminated before $r$-th iteration, where $r$ is the smallest integer satisfying $2^{r-1}\sqrt{n} > m^{*}$. Then, the algorithm can obtain an exact optimal solution to \textbf{(SOCKP)} in pseudo-polynomial time complexity $\mathrm{O}((2^{r}-1) n^{1.5} T) = \mathrm{O}(m^{*} n T)$, where $\mathrm{O}(T)$ is the time complexity of the algorithm for the ordinary BKPs. 

\section{Computational results} \label{sect07}

In this section, we present the computational results of our approximate and exact algorithms for the SOC-constrained BKP. Computational results of solving the mixed-integer quadratically constrained reformulation of \textbf{(SOCKP)} using a commercial optimization software ILOG CPLEX 12.9 are also provided for comparison with our algorithms. Moreover, we compare the quality of the upper bound on the optimal value of \textbf{(SOCKP)} obtained by our approach with the upper bound obtained by the approach used in Han et al. \cite{Han16}. All experiments were performed on an Intel\textcircled{R} Core$^{\mathrm{TM}}$ i5-4670 CPU $@$ 3.40GHz PC with 24GB RAM.

We implemented the algorithms with C++ programming language using Microsoft visual studio 2015 and the Combo algorithm developed by Martello et al. \cite{Martello99} to solve the ordinary BKP. Han et al. \cite{Han16} already observed that the Minknap algorithm developed by Pisinger \cite{Pisinger97} performed better than DP-Profits algorithm \cite{Kellerer04} and the MT1R algorithm \cite{Martello90} in their computational study. The Combo algorithm combines the Minknap algorithm with procedures of generating valid cutting planes and integrating them into the original capacity constraint by surrogate relaxation to improve its performance for large and difficult instances. In the implementation, other parameter values were set as in Martello et al. \cite{Martello99}, but the maximum number of states was set to $10^5$ in consideration of the type and size of the problems. The Combo algorithm requires every parameter to be an integer, but the capacity and weights of ordinary BKPs can be non-integer values in our algorithm. We multiply these parameters by $10^6$, and then round up the weights and round down the capacity. We also used CPLEX with $10^{-6}$ relative MIP gap tolerance in our experiments. 

We generated five types of instances based on the methods of Pisinger \cite{Pisinger97} and Monaci et al. \cite{Monaci13}. Let $\textrm{Unif} [a,b]$ and $\textrm{Unif} \{a,b\}$ denote a continuous uniform random variable in the interval $[a,b]$ and a discrete uniform random variable in a set $\{a, a+1, \cdots, b\}$, respectively. 

First, we generated the following three types of instances:
\begin{enumerate}
\item[-] Strongly correlated instances (SC): each weight $\hat{a}_j$ was randomly generated from $\textrm{Unif} \{1,100\}$ and $p_{j} = \hat{a}_j + 10$.
\item[-] Inverse strongly correlated instances (IC): each profit $p_{j}$ was randomly generated from $\textrm{Unif} \{1,100\}$ and $\hat{a}_{j} = \min \{100, p_{j}+10 \}$.
\item[-] Subset sum instances (SS): each weight $\hat{a}_j$ was randomly generated from $\textrm{Unif} \{1,100\}$ and $p_{j} = \hat{a}_j$.
\end{enumerate}
The standard deviation $\sigma_{j}$ was randomly generated from $\textrm{Unif} [0.05\hat{a}_j, 0.1\hat{a}_j]$ for instances of SC and IC while it was set as $\sigma_{j} = 0.1\hat{a}_{j}$ for SS. Every standard deviation value was rounded off to four decimal places. The capacity of knapsack was set to $b = \lfloor \sum_{j \in {N}} \hat{a}_j / 2\rfloor$.

We also generated instances in which the standard deviations are inversely proportional to the means $\hat{a}_{j}$ using already generated instances. The new standard deviation values are derived as $\hat{\sigma_{j}} := 10.0-\sigma_{j} \times \textrm{Unif}[0.5,0.8]$. The following two types of instances are additionally generated using the new standard deviation $\hat{\sigma_{j}}$ as follows:   
\begin{enumerate}
\item[-] Strongly correlated instances with inversely proportional standard deviations $\hat{\sigma_{j}}$ (SCR)
\item[-] Inverse strongly correlated instances with inversely proportional standard deviations $\hat{\sigma_{j}}$ (ICR)
\end{enumerate}

We set the probability $\rho$ to $0.95$ and $0.99$, whose corresponding values of $\mathrm{\Omega} = \sqrt{\rho/(1-\rho)}$ are about $4.36$ and $9.95$, respectively. We note that the value of $\mathrm{\Omega}= \sqrt{\rho/(1-\rho)}$ becomes greater than $\mathrm{\Phi}^{-1} (\rho)$, whose values are approximately $1.96$ and $2.58$ when the value of $\rho$ is $0.95$ and $0.99$, respectively. 

The size of instances $n$ were $100$, $400$, and $900$ for the comparison between our approach and others in terms of the quality of the upper and lower bounds and the computational performance. We also generated instances whose sizes were $2,500$, $4,900$, and $10,000$ to test our approach on large instances. For each instance type and size, ten instances were generated and tested. We report the average of ten experimental values for each entry in the tables. 

\setlength{\tabcolsep}{1pt}
\renewcommand{\arraystretch}{1.1}
\begin{table}[]
\caption{Computational results of solving \textbf{(RKPm)} to obtain the upper and lower bounds for \textbf{(SOCKP)} simultaneously for different $m$ values.}
\label{table07:01}
\center
\footnotesize
\begin{tabular}{c@{\extracolsep{4pt}}c@{\extracolsep{6pt}}c@{\extracolsep{6pt}}rr@{\extracolsep{6pt}}rr@{\extracolsep{6pt}}rr@{\extracolsep{6pt}}rr@{\extracolsep{4pt}}rr}
\hline 
\multirow{2}{*}{$n$} & \multirow{2}{*}{type} & \multirow{2}{*}{$\rho$} & \multicolumn{2}{c}{$m=\sqrt{n}/2$} & \multicolumn{2}{c}{$m=\sqrt{n}$} & \multicolumn{2}{c}{$m=2\sqrt{n}$} & \multicolumn{2}{c}{$m=3\sqrt{n}$} & \multicolumn{2}{c}{$m=4\sqrt{n}$} \\ \cline{4-5}  \cline{6-7} \cline{8-9} \cline{10-11} \cline{12-13} 
 & & & \multicolumn{1}{c}{time(s)} & \multicolumn{1}{c}{gap(\%)} & \multicolumn{1}{c}{time(s)} & \multicolumn{1}{c}{gap(\%)} & \multicolumn{1}{c}{time(s)} & \multicolumn{1}{c}{gap(\%)} & \multicolumn{1}{c}{time(s)} & \multicolumn{1}{c}{gap(\%)} & \multicolumn{1}{c}{time(s)} & \multicolumn{1}{c}{gap(\%)} \\ \hline 
\multirow{10}{*}{100} & \multirow{2}{*}{SC} & 0.95 & 0.05 & 1.82 & 0.09 & 0.49 & 0.18 & 0.11 & 0.28 & 0.06 & 0.37 & 0.04 \\
 & & 0.99 & 0.04 & 4.05 & 0.08 & 1.00 & 0.17 & 0.24 & 0.25 & 0.11 & 0.35 & 0.07 \\ \cline{2-3} \cline{4-5}  \cline{6-7} \cline{8-9} \cline{10-11} \cline{12-13} 
 & \multirow{2}{*}{SCR} & 0.95 & 0.05 & 3.86 & 0.09 & 0.77 & 0.18 & 0.20 & 0.29 & 0.09 & 0.36 & 0.05 \\
 & & 0.99 & 0.04 & 8.68 & 0.09 & 1.92 & 0.17 & 0.50 & 0.26 & 0.21 & 0.34 & 0.13 \\ \cline{2-3} \cline{4-5}  \cline{6-7} \cline{8-9} \cline{10-11} \cline{12-13} 
 & \multirow{2}{*}{IC} & 0.95 & 0.04 & 2.39 & 0.09 & 0.63 & 0.17 & 0.16 & 0.27 & 0.07 & 0.34 & 0.04 \\
 & & 0.99 & 0.04 & 5.14 & 0.08 & 1.40 & 0.17 & 0.35 & 0.25 & 0.16 & 0.33 & 0.10 \\ \cline{2-3} \cline{4-5}  \cline{6-7} \cline{8-9} \cline{10-11} \cline{12-13} 
 & \multirow{2}{*}{ICR} & 0.95 & 0.06 & 0.66 & 0.11 & 0.18 & 0.22 & 0.05 & 0.35 & 0.03 & 0.45 & 0.01 \\
 & & 0.99 & 0.05 & 1.49 & 0.09 & 0.44 & 0.18 & 0.12 & 0.28 & 0.06 & 0.37 & 0.04 \\ \cline{2-3} \cline{4-5}  \cline{6-7} \cline{8-9} \cline{10-11} \cline{12-13} 
 & \multirow{2}{*}{SS} & 0.95 & 0.05 & 2.86 & 0.09 & 0.70 & 0.18 & 0.18 & 0.29 & 0.08 & 0.36 & 0.05 \\
 & & 0.99 & 0.05 & 6.26 & 0.09 & 1.54 & 0.18 & 0.39 & 0.27 & 0.18 & 0.35 & 0.09 \\ \hline 
\multirow{10}{*}{400} & \multirow{2}{*}{SC} & 0.95 & 0.47 & 0.94 & 0.96 & 0.23 & 2.01 & 0.06 & 3.24 & 0.03 & 4.50 & 0.01 \\
 & & 0.99 & 0.43 & 2.09 & 0.86 & 0.52 & 1.84 & 0.13 & 2.91 & 0.06 & 4.18 & 0.03 \\ \cline{2-3} \cline{4-5}  \cline{6-7} \cline{8-9} \cline{10-11} \cline{12-13}  
 & \multirow{2}{*}{SCR} & 0.95 & 0.51 & 1.82 & 0.99 & 0.38 & 2.09 & 0.10 & 3.27 & 0.05 & 4.69 & 0.02 \\
 & & 0.99 & 0.44 & 4.38 & 0.85 & 0.90 & 1.82 & 0.23 & 2.84 & 0.10 & 4.15 & 0.06 \\ \cline{2-3} \cline{4-5}  \cline{6-7} \cline{8-9} \cline{10-11} \cline{12-13} 
 & \multirow{2}{*}{IC} & 0.95 & 0.42 & 1.24 & 0.86 & 0.32 & 1.80 & 0.08 & 2.84 & 0.03 & 4.18 & 0.02 \\
 & & 0.99 & 0.41 & 2.71 & 0.83 & 0.73 & 1.73 & 0.18 & 2.74 & 0.08 & 4.00 & 0.05 \\ \cline{2-3} \cline{4-5}  \cline{6-7} \cline{8-9} \cline{10-11} \cline{12-13} 
 & \multirow{2}{*}{ICR} & 0.95 & 0.62 & 0.35 & 1.19 & 0.10 & 2.53 & 0.03 & 3.93 & 0.01 & 5.90 & 0.01 \\
 & & 0.99 & 0.50 & 0.79 & 0.93 & 0.22 & 1.97 & 0.06 & 3.09 & 0.02 & 4.96 & 0.02 \\ \cline{2-3} \cline{4-5}  \cline{6-7} \cline{8-9} \cline{10-11} \cline{12-13} 
 & \multirow{2}{*}{SS} & 0.95 & 0.59 & 1.45 & 1.21 & 0.36 & 2.65 & 0.09 & 3.89 & 0.04 & 5.93 & 0.02 \\
 & & 0.99 & 0.61 & 3.25 & 1.27 & 0.81 & 2.72 & 0.21 & 4.02 & 0.09 & 5.83 & 0.05 \\ \hline 
\multirow{10}{*}{900} & \multirow{2}{*}{SC} & 0.95 & 2.15 & 0.63 & 4.23 & 0.16 & 10.27 & 0.04 & 16.52 & 0.02 & 26.61 & 0.01 \\
 & & 0.99 & 1.90 & 1.41 & 3.79 & 0.36 & 9.07 & 0.09 & 16.51 & 0.04 & 23.07 & 0.02 \\ \cline{2-3} \cline{4-5}  \cline{6-7} \cline{8-9} \cline{10-11} \cline{12-13} 
 & \multirow{2}{*}{SCR} & 0.95 & 2.35 & 1.22 & 4.68 & 0.26 & 11.36 & 0.06 & 18.77 & 0.03 & 27.67 & 0.02 \\
 & & 0.99 & 1.92 & 2.88 & 4.05 & 0.59 & 9.52 & 0.16 & 16.30 & 0.07 & 22.74 & 0.04 \\ \cline{2-3} \cline{4-5}  \cline{6-7} \cline{8-9} \cline{10-11} \cline{12-13} 
 & \multirow{2}{*}{IC} & 0.95 & 1.79 & 0.83 & 3.86 & 0.22 & 9.60 & 0.06 & 15.16 & 0.03 & 22.30 & 0.01 \\
 & & 0.99 & 1.71 & 1.87 & 3.73 & 0.50 & 9.10 & 0.13 & 15.35 & 0.06 & 22.07 & 0.03 \\ \cline{2-3} \cline{4-5}  \cline{6-7} \cline{8-9} \cline{10-11} \cline{12-13} 
 & \multirow{2}{*}{ICR} & 0.95 & 2.74 & 0.24 & 5.85 & 0.07 & 13.28 & 0.02 & 21.25 & 0.01 & 29.51 & 0.00 \\
 & & 0.99 & 2.08 & 0.54 & 4.46 & 0.15 & 9.69 & 0.04 & 17.06 & 0.02 & 24.14 & 0.01 \\ \cline{2-3} \cline{4-5}  \cline{6-7} \cline{8-9} \cline{10-11} \cline{12-13} 
 & \multirow{2}{*}{SS} & 0.95 & 3.55 & 0.97 & 7.53 & 0.24 & 15.84 & 0.06 & 26.05 & 0.03 & 36.61 & 0.01 \\
 & & 0.99 & 3.76 & 2.18 & 7.85 & 0.55 & 16.77 & 0.14 & 27.41 & 0.06 & 38.15 & 0.03 \\ \hline 
\end{tabular}
\end{table}

Table \ref{table07:01} presents the computational results of implementing \textbf{(RKPm)} with $\mathrm{\Delta} = m^{2}$ and  $\mathrm{\Delta} = m^{2} + n/4$ to obtain the upper and lower bounds for \textbf{(SOCKP)} simultaneously for each $m=\sqrt{n}/2, \sqrt{n}, \cdots, 4\sqrt{n}$ when $n = 100, 400$, and $900$. We present the computational time of the algorithm in seconds (time(s)) and the percentage gap between the upper and lower bounds ($\textrm{UB}$ and $\textrm{LB}$) on the optimal value of \textbf{(SOCKP)}, i.e. (gap(\%))$ =  (\textrm{UB} - \textrm{LB})/ \textrm{LB} \times 100$. The number of times for solving the ordinary BKPs is approximately equal to $2nm$ to obtain the upper and lower bounds together; hence, they are not reported in Table \ref{table07:01}.

The number of linear segments $m$ needed was at most $2\sqrt{n}$ when $n=100$ and $\sqrt{n}$ when $n=400$ and $n=900$ to obtain the gap between the upper and lower bounds less than $1\%$, while the required computational time was less than 0.2 seconds when $n=100$, 1.5 seconds when $n=400$, and 8 seconds when $n=900$, respectively. Moreover, the number of linear segments needed was at most $4\sqrt{n}$ when $n=100$ and $3\sqrt{n}$ when $n=400$ and $n=900$ to obtain less than $0.1\%$ gap, while the required computational time was less than a second when $n=100$, five seconds when $n=400$, and 30 seconds when $n=900$, respectively. We also observe that the gap tends to increase as the value of $\rho$ increases. 

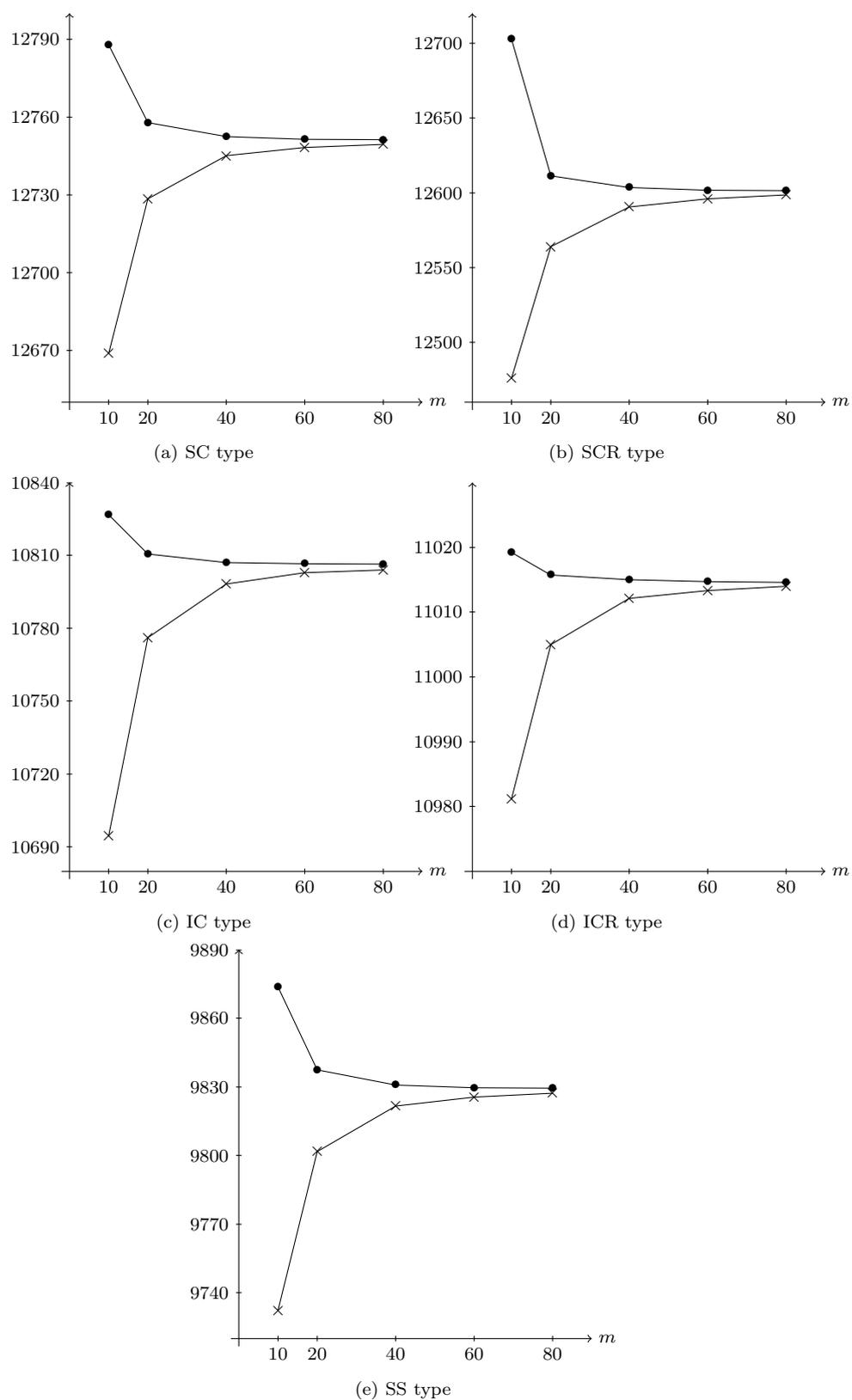
\begin{figure}

\begin{subfigure}{.5\textwidth}
\begin{center}
\begin{tikzpicture}[scale=6]
\draw[->] (-0.02,0) -- (0.90,0) node[right] {$m$};
\draw[->] (0,-0.02) -- (0,1.00) node[above] {};
\foreach \x/\xtext in {0.1/10, 0.2/20, 0.4/40, 0.6/60, 0.8/80}
\draw[shift={(\x,0)}] (0pt,0.2pt) -- (0pt,-0.2pt) node[below] {$\xtext$};
\foreach \y/\ytext in {0.133/12670, 0.333/12700, 0.533/12730, 0.733/12760, 0.933/12790}
\draw[shift={(0,\y)}] (0.2pt,0pt) -- (-0.2pt,0pt) node[left] {$\ytext$};
\draw [] (0.1,0.919) -- (0.2,0.719) -- (0.4,0.683) -- (0.6,0.676) -- (0.8,0.675);  
\draw [] (0.1,0.125) -- (0.2,0.523) -- (0.4,0.634) -- (0.6,0.655) -- (0.8,0.663); 
\draw (0.1,0.919) node[] (A0) {\textbullet}; \draw (0.1,0.125) node[] (B0) {$\mathbf{\times}$};
\draw (0.2,0.719) node[] (A1) {\textbullet}; \draw (0.2,0.523) node[] (B1) {$\mathbf{\times}$};
\draw (0.4,0.683) node[] (A2) {\textbullet}; \draw (0.4,0.634) node[] (B2) {$\mathbf{\times}$};
\draw (0.6,0.676) node[] (A3) {\textbullet}; \draw (0.6,0.655) node[] (B3) {$\mathbf{\times}$};
\draw (0.8,0.675) node[] (A4) {\textbullet}; \draw (0.8,0.663) node[] (B4) {$\mathbf{\times}$};
\end{tikzpicture} 
\caption{SC type}
\label{fig07:06}
\end{center}
\end{subfigure}
\begin{subfigure}{.5\textwidth}
\begin{center}
\begin{tikzpicture}[scale=6]
\draw[->] (-0.02,0) -- (0.90,0) node[right] {$m$};
\draw[->] (0,-0.02) -- (0,1.00) node[above] {};
\foreach \x/\xtext in {0.1/10, 0.2/20, 0.4/40, 0.6/60, 0.8/80}
\draw[shift={(\x,0)}] (0pt,0.2pt) -- (0pt,-0.2pt) node[below] {$\xtext$};
\foreach \y/\ytext in {0.154/12500, 0.346/12550, 0.538/12600, 0.731/12650, 0.923/12700}
\draw[shift={(0,\y)}] (0.2pt,0pt) -- (-0.2pt,0pt) node[left] {$\ytext$};
\draw [] (0.1,0.933) -- (0.2,0.582) -- (0.4,0.552) -- (0.6,0.545) -- (0.8,0.544);  
\draw [] (0.1,0.062) -- (0.2,0.400) -- (0.4,0.502) -- (0.6,0.523) -- (0.8,0.533); 
\draw (0.1,0.933) node[] (A0) {\textbullet}; \draw (0.1,0.062) node[] (B0) {$\mathbf{\times}$};
\draw (0.2,0.582) node[] (A1) {\textbullet}; \draw (0.2,0.400) node[] (B1) {$\mathbf{\times}$};
\draw (0.4,0.552) node[] (A2) {\textbullet}; \draw (0.4,0.502) node[] (B2) {$\mathbf{\times}$};
\draw (0.6,0.545) node[] (A3) {\textbullet}; \draw (0.6,0.523) node[] (B3) {$\mathbf{\times}$};
\draw (0.8,0.544) node[] (A4) {\textbullet}; \draw (0.8,0.533) node[] (B4) {$\mathbf{\times}$};
\end{tikzpicture} 
\caption{SCR type}
\label{fig07:07}
\end{center}
\end{subfigure}
\begin{subfigure}{.5\textwidth}
\begin{center}
\begin{tikzpicture}[scale=6]
\draw[->] (-0.02,0) -- (0.90,0) node[right] {$m$};
\draw[->] (0,-0.02) -- (0,1.00) node[above] {};
\foreach \x/\xtext in {0.1/10, 0.2/20, 0.4/40, 0.6/60, 0.8/80}
\draw[shift={(\x,0)}] (0pt,0.2pt) -- (0pt,-0.2pt) node[below] {$\xtext$};
\foreach \y/\ytext in {0.063/10690, 0.25/10720, 0.438/10750, 0.625/10780, 0.813/10810, 0.999/10840}
\draw[shift={(0,\y)}] (0.2pt,0pt) -- (-0.2pt,0pt) node[left] {$\ytext$};
\draw [] (0.1,0.917) -- (0.2,0.816) -- (0.4,0.794) -- (0.6,0.791) -- (0.8,0.790);  
\draw [] (0.1,0.091) -- (0.2,0.601) -- (0.4,0.739) -- (0.6,0.768) -- (0.8,0.775); 
\draw (0.1,0.917) node[] (A0) {\textbullet}; \draw (0.1,0.091) node[] (B0) {$\mathbf{\times}$};
\draw (0.2,0.816) node[] (A1) {\textbullet}; \draw (0.2,0.601) node[] (B1) {$\mathbf{\times}$};
\draw (0.4,0.794) node[] (A2) {\textbullet}; \draw (0.4,0.739) node[] (B2) {$\mathbf{\times}$};
\draw (0.6,0.791) node[] (A3) {\textbullet}; \draw (0.6,0.768) node[] (B3) {$\mathbf{\times}$};
\draw (0.8,0.790) node[] (A4) {\textbullet}; \draw (0.8,0.775) node[] (B4) {$\mathbf{\times}$};
\end{tikzpicture} 
\caption{IC type}
\label{fig07:08}
\end{center}
\end{subfigure}
\begin{subfigure}{.5\textwidth}
\begin{center}
\begin{tikzpicture}[scale=6]
\draw[->] (-0.02,0) -- (0.90,0) node[right] {$m$};
\draw[->] (0,-0.02) -- (0,1.00) node[above] {};
\foreach \x/\xtext in {0.1/10, 0.2/20, 0.4/40, 0.6/60, 0.8/80}
\draw[shift={(\x,0)}] (0pt,0.2pt) -- (0pt,-0.2pt) node[below] {$\xtext$};
\foreach \y/\ytext in {0.167/10980, 0.333/10990, 0.500/11000, 0.667/11010, 0.833/11020}
\draw[shift={(0,\y)}] (0.2pt,0pt) -- (-0.2pt,0pt) node[left] {$\ytext$};
\draw [] (0.1,0.820) -- (0.2,0.762) -- (0.4,0.750) -- (0.6,0.745) -- (0.8,0.743);  
\draw [] (0.1,0.187) -- (0.2,0.583) -- (0.4,0.702) -- (0.6,0.722) -- (0.8,0.733); 
\draw (0.1,0.820) node[] (A0) {\textbullet}; \draw (0.1,0.187) node[] (B0) {$\mathbf{\times}$};
\draw (0.2,0.762) node[] (A1) {\textbullet}; \draw (0.2,0.583) node[] (B1) {$\mathbf{\times}$};
\draw (0.4,0.750) node[] (A2) {\textbullet}; \draw (0.4,0.702) node[] (B2) {$\mathbf{\times}$};
\draw (0.6,0.745) node[] (A3) {\textbullet}; \draw (0.6,0.722) node[] (B3) {$\mathbf{\times}$};
\draw (0.8,0.743) node[] (A4) {\textbullet}; \draw (0.8,0.733) node[] (B4) {$\mathbf{\times}$};
\end{tikzpicture} 
\caption{ICR type}
\label{fig07:09}
\end{center}
\end{subfigure}
\begin{subfigure}{1.0\textwidth}
\begin{center}
\begin{tikzpicture}[scale=6]
\draw[->] (-0.02,0) -- (0.90,0) node[right] {$m$};
\draw[->] (0,-0.02) -- (0,1.00) node[above] {};
\foreach \x/\xtext in {0.1/10, 0.2/20, 0.4/40, 0.6/60, 0.8/80}
\draw[shift={(\x,0)}] (0pt,0.2pt) -- (0pt,-0.2pt) node[below] {$\xtext$};
\foreach \y/\ytext in {0.118/9740, 0.294/9770, 0.471/9800, 0.647/9830, 0.824/9860, 0.999/9890}
\draw[shift={(0,\y)}] (0.2pt,0pt) -- (-0.2pt,0pt) node[left] {$\ytext$};
\draw [] (0.1,0.904) -- (0.2,0.691) -- (0.4,0.652) -- (0.6,0.645) -- (0.8,0.644);  
\draw [] (0.1,0.073) -- (0.2,0.481) -- (0.4,0.598) -- (0.6,0.621) -- (0.8,0.631); 
\draw (0.1,0.904) node[] (A0) {\textbullet}; \draw (0.1,0.073) node[] (B0) {$\mathbf{\times}$};
\draw (0.2,0.691) node[] (A1) {\textbullet}; \draw (0.2,0.481) node[] (B1) {$\mathbf{\times}$};
\draw (0.4,0.652) node[] (A2) {\textbullet}; \draw (0.4,0.598) node[] (B2) {$\mathbf{\times}$};
\draw (0.6,0.645) node[] (A3) {\textbullet}; \draw (0.6,0.621) node[] (B3) {$\mathbf{\times}$};
\draw (0.8,0.644) node[] (A4) {\textbullet}; \draw (0.8,0.631) node[] (B4) {$\mathbf{\times}$};
\end{tikzpicture} 
\caption{SS type}
\label{fig07:10}
\end{center}
\end{subfigure}

\caption{The upper and lower bounds for \textbf{(SOCKP)}. Each point represents the upper and lower bounds (\textbullet \ and $\mathbf{\times}$), respectively, for $m=10,20,40,60,80$, when $n=400$, $\rho=0.95$.}
\label{fig07_02}
\end{figure}

Figure \ref{fig07_02} shows the change in the upper and lower bound values depending on $m$ for the instances with $n=400$, $\rho=0.95$. We can see that both the upper and lower bounds quickly converge to the optimal value of \textbf{(SOCKP)} as $m$ increases. It can also be observed that the upper bound is closer to the optimal value of \textbf{(SOCKP)} than the lower bound at the same value of $m$.

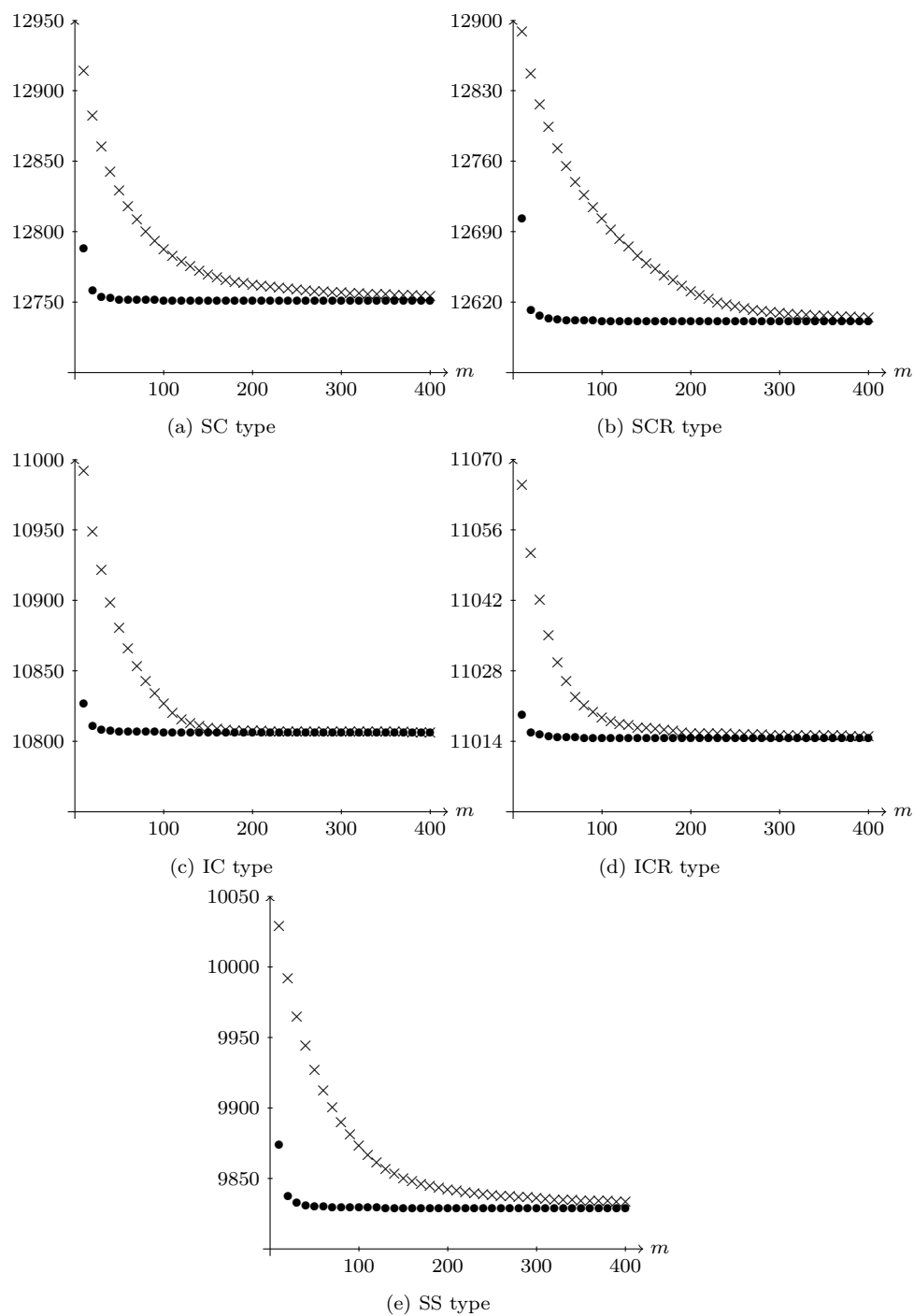
\begin{figure}
\begin{subfigure}{.5\textwidth}
\begin{center}
\begin{tikzpicture}[scale=5]
\draw[->] (-0.02,0) -- (1.05,0) node[right] {$m$};
\draw[->] (0,-0.02) -- (0,1.00) node[above] {};
\foreach \x/\xtext in {0.25/100, 0.5/200, 0.75/300, 1.0/400}
\draw[shift={(\x,0)}] (0pt,0.2pt) -- (0pt,-0.2pt) node[below] {$\xtext$};
\foreach \y/\ytext in {0.2/12750, 0.4/12800, 0.6/12850, 0.8/12900, 0.999/12950}
\draw[shift={(0,\y)}] (0.2pt,0pt) -- (-0.2pt,0pt) node[left] {$\ytext$};

\draw ( 0.025 , 0.3512 ) node[] ( A042 ) {\textbullet};    \draw ( 0.025 , 0.8560 ) node[] ( B042 ) {$\mathbf{\times}$};
\draw ( 0.05 , 0.2312 ) node[] ( A043 ) {\textbullet};    \draw ( 0.05 , 0.7300 ) node[] ( B043 ) {$\mathbf{\times}$};
\draw ( 0.075 , 0.2148 ) node[] ( A044 ) {\textbullet};    \draw ( 0.075 , 0.6404 ) node[] ( B044 ) {$\mathbf{\times}$};
\draw ( 0.1 , 0.2096 ) node[] ( A045 ) {\textbullet};    \draw ( 0.1 , 0.5712 ) node[] ( B045 ) {$\mathbf{\times}$};
\draw ( 0.125 , 0.2060 ) node[] ( A046 ) {\textbullet};    \draw ( 0.125 , 0.5164 ) node[] ( B046 ) {$\mathbf{\times}$};
\draw ( 0.15 , 0.2056 ) node[] ( A047 ) {\textbullet};    \draw ( 0.15 , 0.4716 ) node[] ( B047 ) {$\mathbf{\times}$};
\draw ( 0.175 , 0.2048 ) node[] ( A048 ) {\textbullet};    \draw ( 0.175 , 0.4340 ) node[] ( B048 ) {$\mathbf{\times}$};
\draw ( 0.2 , 0.2048 ) node[] ( A049 ) {\textbullet};    \draw ( 0.2 , 0.4012 ) node[] ( B049 ) {$\mathbf{\times}$};
\draw ( 0.225 , 0.2044 ) node[] ( A050 ) {\textbullet};    \draw ( 0.225 , 0.3748 ) node[] ( B050 ) {$\mathbf{\times}$};
\draw ( 0.25 , 0.2036 ) node[] ( A051 ) {\textbullet};    \draw ( 0.25 , 0.3512 ) node[] ( B051 ) {$\mathbf{\times}$};
\draw ( 0.275 , 0.2028 ) node[] ( A052 ) {\textbullet};    \draw ( 0.275 , 0.3320 ) node[] ( B052 ) {$\mathbf{\times}$};
\draw ( 0.3 , 0.2028 ) node[] ( A053 ) {\textbullet};    \draw ( 0.3 , 0.3160 ) node[] ( B053 ) {$\mathbf{\times}$};
\draw ( 0.325 , 0.2028 ) node[] ( A054 ) {\textbullet};    \draw ( 0.325 , 0.3012 ) node[] ( B054 ) {$\mathbf{\times}$};
\draw ( 0.35 , 0.2028 ) node[] ( A055 ) {\textbullet};    \draw ( 0.35 , 0.2892 ) node[] ( B055 ) {$\mathbf{\times}$};
\draw ( 0.375 , 0.2028 ) node[] ( A056 ) {\textbullet};    \draw ( 0.375 , 0.2772 ) node[] ( B056 ) {$\mathbf{\times}$};
\draw ( 0.4 , 0.2024 ) node[] ( A057 ) {\textbullet};    \draw ( 0.4 , 0.2696 ) node[] ( B057 ) {$\mathbf{\times}$};
\draw ( 0.425 , 0.2020 ) node[] ( A058 ) {\textbullet};    \draw ( 0.425 , 0.2632 ) node[] ( B058 ) {$\mathbf{\times}$};
\draw ( 0.45 , 0.2020 ) node[] ( A059 ) {\textbullet};    \draw ( 0.45 , 0.2580 ) node[] ( B059 ) {$\mathbf{\times}$};
\draw ( 0.475 , 0.2020 ) node[] ( A060 ) {\textbullet};    \draw ( 0.475 , 0.2536 ) node[] ( B060 ) {$\mathbf{\times}$};
\draw ( 0.5 , 0.2020 ) node[] ( A061 ) {\textbullet};    \draw ( 0.5 , 0.2496 ) node[] ( B061 ) {$\mathbf{\times}$};
\draw ( 0.525 , 0.2020 ) node[] ( A062 ) {\textbullet};    \draw ( 0.525 , 0.2464 ) node[] ( B062 ) {$\mathbf{\times}$};
\draw ( 0.55 , 0.2020 ) node[] ( A063 ) {\textbullet};    \draw ( 0.55 , 0.2432 ) node[] ( B063 ) {$\mathbf{\times}$};
\draw ( 0.575 , 0.2020 ) node[] ( A064 ) {\textbullet};    \draw ( 0.575 , 0.2400 ) node[] ( B064 ) {$\mathbf{\times}$};
\draw ( 0.6 , 0.2020 ) node[] ( A065 ) {\textbullet};    \draw ( 0.6 , 0.2376 ) node[] ( B065 ) {$\mathbf{\times}$};
\draw ( 0.625 , 0.2020 ) node[] ( A066 ) {\textbullet};    \draw ( 0.625 , 0.2356 ) node[] ( B066 ) {$\mathbf{\times}$};
\draw ( 0.65 , 0.2020 ) node[] ( A067 ) {\textbullet};    \draw ( 0.65 , 0.2324 ) node[] ( B067 ) {$\mathbf{\times}$};
\draw ( 0.675 , 0.2020 ) node[] ( A068 ) {\textbullet};    \draw ( 0.675 , 0.2316 ) node[] ( B068 ) {$\mathbf{\times}$};
\draw ( 0.7 , 0.2020 ) node[] ( A069 ) {\textbullet};    \draw ( 0.7 , 0.2296 ) node[] ( B069 ) {$\mathbf{\times}$};
\draw ( 0.725 , 0.2020 ) node[] ( A070 ) {\textbullet};    \draw ( 0.725 , 0.2284 ) node[] ( B070 ) {$\mathbf{\times}$};
\draw ( 0.75 , 0.2020 ) node[] ( A071 ) {\textbullet};    \draw ( 0.75 , 0.2268 ) node[] ( B071 ) {$\mathbf{\times}$};
\draw ( 0.775 , 0.2020 ) node[] ( A072 ) {\textbullet};    \draw ( 0.775 , 0.2252 ) node[] ( B072 ) {$\mathbf{\times}$};
\draw ( 0.8 , 0.2020 ) node[] ( A073 ) {\textbullet};    \draw ( 0.8 , 0.2244 ) node[] ( B073 ) {$\mathbf{\times}$};
\draw ( 0.825 , 0.2020 ) node[] ( A074 ) {\textbullet};    \draw ( 0.825 , 0.2236 ) node[] ( B074 ) {$\mathbf{\times}$};
\draw ( 0.85 , 0.2020 ) node[] ( A075 ) {\textbullet};    \draw ( 0.85 , 0.2228 ) node[] ( B075 ) {$\mathbf{\times}$};
\draw ( 0.875 , 0.2020 ) node[] ( A076 ) {\textbullet};    \draw ( 0.875 , 0.2220 ) node[] ( B076 ) {$\mathbf{\times}$};
\draw ( 0.9 , 0.2020 ) node[] ( A077 ) {\textbullet};    \draw ( 0.9 , 0.2208 ) node[] ( B077 ) {$\mathbf{\times}$};
\draw ( 0.925 , 0.2020 ) node[] ( A078 ) {\textbullet};    \draw ( 0.925 , 0.2196 ) node[] ( B078 ) {$\mathbf{\times}$};
\draw ( 0.95 , 0.2020 ) node[] ( A079 ) {\textbullet};    \draw ( 0.95 , 0.2188 ) node[] ( B079 ) {$\mathbf{\times}$};
\draw ( 0.975 , 0.2020 ) node[] ( A080 ) {\textbullet};    \draw ( 0.975 , 0.2184 ) node[] ( B080 ) {$\mathbf{\times}$};
\draw ( 1 , 0.2020 ) node[] ( A081 ) {\textbullet};    \draw ( 1 , 0.2184 ) node[] ( B081 ) {$\mathbf{\times}$};

\end{tikzpicture} 
\caption{SC type}
\label{fig07:01}
\end{center}
\end{subfigure}
\begin{subfigure}{.5\textwidth}
\begin{center}
\begin{tikzpicture}[scale=5]
\draw[->] (-0.02,0) -- (1.05,0) node[right] {$m$};
\draw[->] (0,-0.02) -- (0,1.00) node[above] {};
\foreach \x/\xtext in {0.25/100, 0.5/200, 0.75/300, 1.0/400}
\draw[shift={(\x,0)}] (0pt,0.2pt) -- (0pt,-0.2pt) node[below] {$\xtext$};
\foreach \y/\ytext in {0.2/12620, 0.4/12690, 0.6/12760, 0.8/12830, 0.999/12900}
\draw[shift={(0,\y)}] (0.2pt,0pt) -- (-0.2pt,0pt) node[left] {$\ytext$};
\draw ( 0.025 , 0.4363 ) node[] ( A042 ) {\textbullet};    \draw ( 0.025 , 0.9689 ) node[] ( B042 ) {$\mathbf{\times}$};
\draw ( 0.05 , 0.1754 ) node[] ( A043 ) {\textbullet};    \draw ( 0.05 , 0.8491 ) node[] ( B043 ) {$\mathbf{\times}$};
\draw ( 0.075 , 0.1600 ) node[] ( A044 ) {\textbullet};    \draw ( 0.075 , 0.7597 ) node[] ( B044 ) {$\mathbf{\times}$};
\draw ( 0.1 , 0.1526 ) node[] ( A045 ) {\textbullet};    \draw ( 0.1 , 0.6963 ) node[] ( B045 ) {$\mathbf{\times}$};
\draw ( 0.125 , 0.1497 ) node[] ( A046 ) {\textbullet};    \draw ( 0.125 , 0.6354 ) node[] ( B046 ) {$\mathbf{\times}$};
\draw ( 0.15 , 0.1477 ) node[] ( A047 ) {\textbullet};    \draw ( 0.15 , 0.5857 ) node[] ( B047 ) {$\mathbf{\times}$};
\draw ( 0.175 , 0.1477 ) node[] ( A048 ) {\textbullet};    \draw ( 0.175 , 0.5411 ) node[] ( B048 ) {$\mathbf{\times}$};
\draw ( 0.2 , 0.1469 ) node[] ( A049 ) {\textbullet};    \draw ( 0.2 , 0.5043 ) node[] ( B049 ) {$\mathbf{\times}$};
\draw ( 0.225 , 0.1463 ) node[] ( A050 ) {\textbullet};    \draw ( 0.225 , 0.4703 ) node[] ( B050 ) {$\mathbf{\times}$};
\draw ( 0.25 , 0.1454 ) node[] ( A051 ) {\textbullet};    \draw ( 0.25 , 0.4363 ) node[] ( B051 ) {$\mathbf{\times}$};
\draw ( 0.275 , 0.1451 ) node[] ( A052 ) {\textbullet};    \draw ( 0.275 , 0.4054 ) node[] ( B052 ) {$\mathbf{\times}$};
\draw ( 0.3 , 0.1451 ) node[] ( A053 ) {\textbullet};    \draw ( 0.3 , 0.3789 ) node[] ( B053 ) {$\mathbf{\times}$};
\draw ( 0.325 , 0.1451 ) node[] ( A054 ) {\textbullet};    \draw ( 0.325 , 0.3574 ) node[] ( B054 ) {$\mathbf{\times}$};
\draw ( 0.35 , 0.1451 ) node[] ( A055 ) {\textbullet};    \draw ( 0.35 , 0.3323 ) node[] ( B055 ) {$\mathbf{\times}$};
\draw ( 0.375 , 0.1449 ) node[] ( A056 ) {\textbullet};    \draw ( 0.375 , 0.3106 ) node[] ( B056 ) {$\mathbf{\times}$};
\draw ( 0.4 , 0.1449 ) node[] ( A057 ) {\textbullet};    \draw ( 0.4 , 0.2937 ) node[] ( B057 ) {$\mathbf{\times}$};
\draw ( 0.425 , 0.1449 ) node[] ( A058 ) {\textbullet};    \draw ( 0.425 , 0.2757 ) node[] ( B058 ) {$\mathbf{\times}$};
\draw ( 0.45 , 0.1449 ) node[] ( A059 ) {\textbullet};    \draw ( 0.45 , 0.2611 ) node[] ( B059 ) {$\mathbf{\times}$};
\draw ( 0.475 , 0.1449 ) node[] ( A060 ) {\textbullet};    \draw ( 0.475 , 0.2463 ) node[] ( B060 ) {$\mathbf{\times}$};
\draw ( 0.5 , 0.1449 ) node[] ( A061 ) {\textbullet};    \draw ( 0.5 , 0.2317 ) node[] ( B061 ) {$\mathbf{\times}$};
\draw ( 0.525 , 0.1449 ) node[] ( A062 ) {\textbullet};    \draw ( 0.525 , 0.2209 ) node[] ( B062 ) {$\mathbf{\times}$};
\draw ( 0.55 , 0.1449 ) node[] ( A063 ) {\textbullet};    \draw ( 0.55 , 0.2089 ) node[] ( B063 ) {$\mathbf{\times}$};
\draw ( 0.575 , 0.1449 ) node[] ( A064 ) {\textbullet};    \draw ( 0.575 , 0.1994 ) node[] ( B064 ) {$\mathbf{\times}$};
\draw ( 0.6 , 0.1449 ) node[] ( A065 ) {\textbullet};    \draw ( 0.6 , 0.1929 ) node[] ( B065 ) {$\mathbf{\times}$};
\draw ( 0.625 , 0.1449 ) node[] ( A066 ) {\textbullet};    \draw ( 0.625 , 0.1869 ) node[] ( B066 ) {$\mathbf{\times}$};
\draw ( 0.65 , 0.1449 ) node[] ( A067 ) {\textbullet};    \draw ( 0.65 , 0.1823 ) node[] ( B067 ) {$\mathbf{\times}$};
\draw ( 0.675 , 0.1449 ) node[] ( A068 ) {\textbullet};    \draw ( 0.675 , 0.1771 ) node[] ( B068 ) {$\mathbf{\times}$};
\draw ( 0.7 , 0.1449 ) node[] ( A069 ) {\textbullet};    \draw ( 0.7 , 0.1737 ) node[] ( B069 ) {$\mathbf{\times}$};
\draw ( 0.725 , 0.1449 ) node[] ( A070 ) {\textbullet};    \draw ( 0.725 , 0.1709 ) node[] ( B070 ) {$\mathbf{\times}$};
\draw ( 0.75 , 0.1449 ) node[] ( A071 ) {\textbullet};    \draw ( 0.75 , 0.1683 ) node[] ( B071 ) {$\mathbf{\times}$};
\draw ( 0.775 , 0.1449 ) node[] ( A072 ) {\textbullet};    \draw ( 0.775 , 0.1663 ) node[] ( B072 ) {$\mathbf{\times}$};
\draw ( 0.8 , 0.1449 ) node[] ( A073 ) {\textbullet};    \draw ( 0.8 , 0.1649 ) node[] ( B073 ) {$\mathbf{\times}$};
\draw ( 0.825 , 0.1449 ) node[] ( A074 ) {\textbullet};    \draw ( 0.825 , 0.1634 ) node[] ( B074 ) {$\mathbf{\times}$};
\draw ( 0.85 , 0.1449 ) node[] ( A075 ) {\textbullet};    \draw ( 0.85 , 0.1623 ) node[] ( B075 ) {$\mathbf{\times}$};
\draw ( 0.875 , 0.1449 ) node[] ( A076 ) {\textbullet};    \draw ( 0.875 , 0.1606 ) node[] ( B076 ) {$\mathbf{\times}$};
\draw ( 0.9 , 0.1449 ) node[] ( A077 ) {\textbullet};    \draw ( 0.9 , 0.1600 ) node[] ( B077 ) {$\mathbf{\times}$};
\draw ( 0.925 , 0.1449 ) node[] ( A078 ) {\textbullet};    \draw ( 0.925 , 0.1586 ) node[] ( B078 ) {$\mathbf{\times}$};
\draw ( 0.95 , 0.1449 ) node[] ( A079 ) {\textbullet};    \draw ( 0.95 , 0.1577 ) node[] ( B079 ) {$\mathbf{\times}$};
\draw ( 0.975 , 0.1449 ) node[] ( A080 ) {\textbullet};    \draw ( 0.975 , 0.1574 ) node[] ( B080 ) {$\mathbf{\times}$};
\draw ( 1 , 0.1449 ) node[] ( A081 ) {\textbullet};    \draw ( 1 , 0.1566 ) node[] ( B081 ) {$\mathbf{\times}$};

\end{tikzpicture} 
\caption{SCR type}
\label{fig07:02}
\end{center}
\end{subfigure}
\begin{subfigure}{.5\textwidth}
\begin{center}
\begin{tikzpicture}[scale=5]
\draw[->] (-0.02,0) -- (1.05,0) node[right] {$m$};
\draw[->] (0,-0.02) -- (0,1.00) node[above] {};
\foreach \x/\xtext in {0.25/100, 0.5/200, 0.75/300, 1.0/400}
\draw[shift={(\x,0)}] (0pt,0.2pt) -- (0pt,-0.2pt) node[below] {$\xtext$};
\foreach \y/\ytext in {0.2/10800, 0.4/10850, 0.6/10900, 0.8/10950, 0.999/11000}
\draw[shift={(0,\y)}] (0.2pt,0pt) -- (-0.2pt,0pt) node[left] {$\ytext$};

\draw ( 0.025 , 0.3068 ) node[] ( A042 ) {\textbullet};    \draw ( 0.025 , 0.9688 ) node[] ( B042 ) {$\mathbf{\times}$};
\draw ( 0.05 , 0.2420 ) node[] ( A043 ) {\textbullet};    \draw ( 0.05 , 0.7956 ) node[] ( B043 ) {$\mathbf{\times}$};
\draw ( 0.075 , 0.2308 ) node[] ( A044 ) {\textbullet};    \draw ( 0.075 , 0.6856 ) node[] ( B044 ) {$\mathbf{\times}$};
\draw ( 0.1 , 0.2280 ) node[] ( A045 ) {\textbullet};    \draw ( 0.1 , 0.5944 ) node[] ( B045 ) {$\mathbf{\times}$};
\draw ( 0.125 , 0.2268 ) node[] ( A046 ) {\textbullet};    \draw ( 0.125 , 0.5212 ) node[] ( B046 ) {$\mathbf{\times}$};
\draw ( 0.15 , 0.2264 ) node[] ( A047 ) {\textbullet};    \draw ( 0.15 , 0.4640 ) node[] ( B047 ) {$\mathbf{\times}$};
\draw ( 0.175 , 0.2260 ) node[] ( A048 ) {\textbullet};    \draw ( 0.175 , 0.4144 ) node[] ( B048 ) {$\mathbf{\times}$};
\draw ( 0.2 , 0.2256 ) node[] ( A049 ) {\textbullet};    \draw ( 0.2 , 0.3716 ) node[] ( B049 ) {$\mathbf{\times}$};
\draw ( 0.225 , 0.2256 ) node[] ( A050 ) {\textbullet};    \draw ( 0.225 , 0.3368 ) node[] ( B050 ) {$\mathbf{\times}$};
\draw ( 0.25 , 0.2252 ) node[] ( A051 ) {\textbullet};    \draw ( 0.25 , 0.3068 ) node[] ( B051 ) {$\mathbf{\times}$};
\draw ( 0.275 , 0.2248 ) node[] ( A052 ) {\textbullet};    \draw ( 0.275 , 0.2816 ) node[] ( B052 ) {$\mathbf{\times}$};
\draw ( 0.3 , 0.2244 ) node[] ( A053 ) {\textbullet};    \draw ( 0.3 , 0.2620 ) node[] ( B053 ) {$\mathbf{\times}$};
\draw ( 0.325 , 0.2240 ) node[] ( A054 ) {\textbullet};    \draw ( 0.325 , 0.2504 ) node[] ( B054 ) {$\mathbf{\times}$};
\draw ( 0.35 , 0.2240 ) node[] ( A055 ) {\textbullet};    \draw ( 0.35 , 0.2428 ) node[] ( B055 ) {$\mathbf{\times}$};
\draw ( 0.375 , 0.2240 ) node[] ( A056 ) {\textbullet};    \draw ( 0.375 , 0.2376 ) node[] ( B056 ) {$\mathbf{\times}$};
\draw ( 0.4 , 0.2240 ) node[] ( A057 ) {\textbullet};    \draw ( 0.4 , 0.2348 ) node[] ( B057 ) {$\mathbf{\times}$};
\draw ( 0.425 , 0.2240 ) node[] ( A058 ) {\textbullet};    \draw ( 0.425 , 0.2328 ) node[] ( B058 ) {$\mathbf{\times}$};
\draw ( 0.45 , 0.2240 ) node[] ( A059 ) {\textbullet};    \draw ( 0.45 , 0.2316 ) node[] ( B059 ) {$\mathbf{\times}$};
\draw ( 0.475 , 0.2240 ) node[] ( A060 ) {\textbullet};    \draw ( 0.475 , 0.2308 ) node[] ( B060 ) {$\mathbf{\times}$};
\draw ( 0.5 , 0.2236 ) node[] ( A061 ) {\textbullet};    \draw ( 0.5 , 0.2300 ) node[] ( B061 ) {$\mathbf{\times}$};
\draw ( 0.525 , 0.2236 ) node[] ( A062 ) {\textbullet};    \draw ( 0.525 , 0.2292 ) node[] ( B062 ) {$\mathbf{\times}$};
\draw ( 0.55 , 0.2240 ) node[] ( A063 ) {\textbullet};    \draw ( 0.55 , 0.2284 ) node[] ( B063 ) {$\mathbf{\times}$};
\draw ( 0.575 , 0.2236 ) node[] ( A064 ) {\textbullet};    \draw ( 0.575 , 0.2272 ) node[] ( B064 ) {$\mathbf{\times}$};
\draw ( 0.6 , 0.2236 ) node[] ( A065 ) {\textbullet};    \draw ( 0.6 , 0.2272 ) node[] ( B065 ) {$\mathbf{\times}$};
\draw ( 0.625 , 0.2236 ) node[] ( A066 ) {\textbullet};    \draw ( 0.625 , 0.2272 ) node[] ( B066 ) {$\mathbf{\times}$};
\draw ( 0.65 , 0.2236 ) node[] ( A067 ) {\textbullet};    \draw ( 0.65 , 0.2268 ) node[] ( B067 ) {$\mathbf{\times}$};
\draw ( 0.675 , 0.2236 ) node[] ( A068 ) {\textbullet};    \draw ( 0.675 , 0.2268 ) node[] ( B068 ) {$\mathbf{\times}$};
\draw ( 0.7 , 0.2236 ) node[] ( A069 ) {\textbullet};    \draw ( 0.7 , 0.2268 ) node[] ( B069 ) {$\mathbf{\times}$};
\draw ( 0.725 , 0.2236 ) node[] ( A070 ) {\textbullet};    \draw ( 0.725 , 0.2268 ) node[] ( B070 ) {$\mathbf{\times}$};
\draw ( 0.75 , 0.2236 ) node[] ( A071 ) {\textbullet};    \draw ( 0.75 , 0.2268 ) node[] ( B071 ) {$\mathbf{\times}$};
\draw ( 0.775 , 0.2236 ) node[] ( A072 ) {\textbullet};    \draw ( 0.775 , 0.2268 ) node[] ( B072 ) {$\mathbf{\times}$};
\draw ( 0.8 , 0.2236 ) node[] ( A073 ) {\textbullet};    \draw ( 0.8 , 0.2268 ) node[] ( B073 ) {$\mathbf{\times}$};
\draw ( 0.825 , 0.2236 ) node[] ( A074 ) {\textbullet};    \draw ( 0.825 , 0.2268 ) node[] ( B074 ) {$\mathbf{\times}$};
\draw ( 0.85 , 0.2236 ) node[] ( A075 ) {\textbullet};    \draw ( 0.85 , 0.2268 ) node[] ( B075 ) {$\mathbf{\times}$};
\draw ( 0.875 , 0.2236 ) node[] ( A076 ) {\textbullet};    \draw ( 0.875 , 0.2268 ) node[] ( B076 ) {$\mathbf{\times}$};
\draw ( 0.9 , 0.2236 ) node[] ( A077 ) {\textbullet};    \draw ( 0.9 , 0.2268 ) node[] ( B077 ) {$\mathbf{\times}$};
\draw ( 0.925 , 0.2236 ) node[] ( A078 ) {\textbullet};    \draw ( 0.925 , 0.2268 ) node[] ( B078 ) {$\mathbf{\times}$};
\draw ( 0.95 , 0.2236 ) node[] ( A079 ) {\textbullet};    \draw ( 0.95 , 0.2260 ) node[] ( B079 ) {$\mathbf{\times}$};
\draw ( 0.975 , 0.2236 ) node[] ( A080 ) {\textbullet};    \draw ( 0.975 , 0.2260 ) node[] ( B080 ) {$\mathbf{\times}$};
\draw ( 1 , 0.2236 ) node[] ( A081 ) {\textbullet};    \draw ( 1 , 0.2260 ) node[] ( B081 ) {$\mathbf{\times}$};

\end{tikzpicture} 
\caption{IC type}
\label{fig07:03}
\end{center}
\end{subfigure}
\begin{subfigure}{.5\textwidth}
\begin{center}
\begin{tikzpicture}[scale=5]
\draw[->] (-0.02,0) -- (1.05,0) node[right] {$m$};
\draw[->] (0,-0.02) -- (0,1.00) node[above] {};
\foreach \x/\xtext in {0.25/100, 0.5/200, 0.75/300, 1.0/400}
\draw[shift={(\x,0)}] (0pt,0.2pt) -- (0pt,-0.2pt) node[below] {$\xtext$};
\foreach \y/\ytext in {0.2/11014, 0.4/11028, 0.6/11042, 0.8/11056, 1.0/11070}
\draw[shift={(0,\y)}] (0.2pt,0pt) -- (-0.2pt,0pt) node[left] {$\ytext$};
\draw ( 0.025 , 0.2743 ) node[] ( A042 ) {\textbullet};    \draw ( 0.025 , 0.9271 ) node[] ( B042 ) {$\mathbf{\times}$};
\draw ( 0.05 , 0.2243 ) node[] ( A043 ) {\textbullet};    \draw ( 0.05 , 0.7343 ) node[] ( B043 ) {$\mathbf{\times}$};
\draw ( 0.075 , 0.2186 ) node[] ( A044 ) {\textbullet};    \draw ( 0.075 , 0.6014 ) node[] ( B044 ) {$\mathbf{\times}$};
\draw ( 0.1 , 0.2143 ) node[] ( A045 ) {\textbullet};    \draw ( 0.1 , 0.5000 ) node[] ( B045 ) {$\mathbf{\times}$};
\draw ( 0.125 , 0.2114 ) node[] ( A046 ) {\textbullet};    \draw ( 0.125 , 0.4243 ) node[] ( B046 ) {$\mathbf{\times}$};
\draw ( 0.15 , 0.2100 ) node[] ( A047 ) {\textbullet};    \draw ( 0.15 , 0.3700 ) node[] ( B047 ) {$\mathbf{\times}$};
\draw ( 0.175 , 0.2100 ) node[] ( A048 ) {\textbullet};    \draw ( 0.175 , 0.3271 ) node[] ( B048 ) {$\mathbf{\times}$};
\draw ( 0.2 , 0.2086 ) node[] ( A049 ) {\textbullet};    \draw ( 0.2 , 0.3014 ) node[] ( B049 ) {$\mathbf{\times}$};
\draw ( 0.225 , 0.2086 ) node[] ( A050 ) {\textbullet};    \draw ( 0.225 , 0.2829 ) node[] ( B050 ) {$\mathbf{\times}$};
\draw ( 0.25 , 0.2086 ) node[] ( A051 ) {\textbullet};    \draw ( 0.25 , 0.2686 ) node[] ( B051 ) {$\mathbf{\times}$};
\draw ( 0.275 , 0.2086 ) node[] ( A052 ) {\textbullet};    \draw ( 0.275 , 0.2571 ) node[] ( B052 ) {$\mathbf{\times}$};
\draw ( 0.3 , 0.2086 ) node[] ( A053 ) {\textbullet};    \draw ( 0.3 , 0.2500 ) node[] ( B053 ) {$\mathbf{\times}$};
\draw ( 0.325 , 0.2086 ) node[] ( A054 ) {\textbullet};    \draw ( 0.325 , 0.2471 ) node[] ( B054 ) {$\mathbf{\times}$};
\draw ( 0.35 , 0.2086 ) node[] ( A055 ) {\textbullet};    \draw ( 0.35 , 0.2371 ) node[] ( B055 ) {$\mathbf{\times}$};
\draw ( 0.375 , 0.2086 ) node[] ( A056 ) {\textbullet};    \draw ( 0.375 , 0.2371 ) node[] ( B056 ) {$\mathbf{\times}$};
\draw ( 0.4 , 0.2086 ) node[] ( A057 ) {\textbullet};    \draw ( 0.4 , 0.2357 ) node[] ( B057 ) {$\mathbf{\times}$};
\draw ( 0.425 , 0.2086 ) node[] ( A058 ) {\textbullet};    \draw ( 0.425 , 0.2329 ) node[] ( B058 ) {$\mathbf{\times}$};
\draw ( 0.45 , 0.2086 ) node[] ( A059 ) {\textbullet};    \draw ( 0.45 , 0.2300 ) node[] ( B059 ) {$\mathbf{\times}$};
\draw ( 0.475 , 0.2086 ) node[] ( A060 ) {\textbullet};    \draw ( 0.475 , 0.2243 ) node[] ( B060 ) {$\mathbf{\times}$};
\draw ( 0.5 , 0.2086 ) node[] ( A061 ) {\textbullet};    \draw ( 0.5 , 0.2229 ) node[] ( B061 ) {$\mathbf{\times}$};
\draw ( 0.525 , 0.2086 ) node[] ( A062 ) {\textbullet};    \draw ( 0.525 , 0.2229 ) node[] ( B062 ) {$\mathbf{\times}$};
\draw ( 0.55 , 0.2086 ) node[] ( A063 ) {\textbullet};    \draw ( 0.55 , 0.2229 ) node[] ( B063 ) {$\mathbf{\times}$};
\draw ( 0.575 , 0.2086 ) node[] ( A064 ) {\textbullet};    \draw ( 0.575 , 0.2229 ) node[] ( B064 ) {$\mathbf{\times}$};
\draw ( 0.6 , 0.2086 ) node[] ( A065 ) {\textbullet};    \draw ( 0.6 , 0.2229 ) node[] ( B065 ) {$\mathbf{\times}$};
\draw ( 0.625 , 0.2086 ) node[] ( A066 ) {\textbullet};    \draw ( 0.625 , 0.2200 ) node[] ( B066 ) {$\mathbf{\times}$};
\draw ( 0.65 , 0.2086 ) node[] ( A067 ) {\textbullet};    \draw ( 0.65 , 0.2200 ) node[] ( B067 ) {$\mathbf{\times}$};
\draw ( 0.675 , 0.2086 ) node[] ( A068 ) {\textbullet};    \draw ( 0.675 , 0.2200 ) node[] ( B068 ) {$\mathbf{\times}$};
\draw ( 0.7 , 0.2086 ) node[] ( A069 ) {\textbullet};    \draw ( 0.7 , 0.2186 ) node[] ( B069 ) {$\mathbf{\times}$};
\draw ( 0.725 , 0.2086 ) node[] ( A070 ) {\textbullet};    \draw ( 0.725 , 0.2171 ) node[] ( B070 ) {$\mathbf{\times}$};
\draw ( 0.75 , 0.2086 ) node[] ( A071 ) {\textbullet};    \draw ( 0.75 , 0.2171 ) node[] ( B071 ) {$\mathbf{\times}$};
\draw ( 0.775 , 0.2086 ) node[] ( A072 ) {\textbullet};    \draw ( 0.775 , 0.2171 ) node[] ( B072 ) {$\mathbf{\times}$};
\draw ( 0.8 , 0.2086 ) node[] ( A073 ) {\textbullet};    \draw ( 0.8 , 0.2171 ) node[] ( B073 ) {$\mathbf{\times}$};
\draw ( 0.825 , 0.2086 ) node[] ( A074 ) {\textbullet};    \draw ( 0.825 , 0.2171 ) node[] ( B074 ) {$\mathbf{\times}$};
\draw ( 0.85 , 0.2086 ) node[] ( A075 ) {\textbullet};    \draw ( 0.85 , 0.2171 ) node[] ( B075 ) {$\mathbf{\times}$};
\draw ( 0.875 , 0.2086 ) node[] ( A076 ) {\textbullet};    \draw ( 0.875 , 0.2171 ) node[] ( B076 ) {$\mathbf{\times}$};
\draw ( 0.9 , 0.2086 ) node[] ( A077 ) {\textbullet};    \draw ( 0.9 , 0.2171 ) node[] ( B077 ) {$\mathbf{\times}$};
\draw ( 0.925 , 0.2086 ) node[] ( A078 ) {\textbullet};    \draw ( 0.925 , 0.2171 ) node[] ( B078 ) {$\mathbf{\times}$};
\draw ( 0.95 , 0.2086 ) node[] ( A079 ) {\textbullet};    \draw ( 0.95 , 0.2157 ) node[] ( B079 ) {$\mathbf{\times}$};
\draw ( 0.975 , 0.2086 ) node[] ( A080 ) {\textbullet};    \draw ( 0.975 , 0.2143 ) node[] ( B080 ) {$\mathbf{\times}$};
\draw ( 1 , 0.2086 ) node[] ( A081 ) {\textbullet};    \draw ( 1 , 0.2143 ) node[] ( B081 ) {$\mathbf{\times}$};

\end{tikzpicture} 
\caption{ICR type}
\label{fig07:04}
\end{center}
\end{subfigure}
\begin{subfigure}{1.0\textwidth}
\begin{center}
\begin{tikzpicture}[scale=5]
\draw[->] (-0.02,0) -- (1.05,0) node[right] {$m$};
\draw[->] (0,-0.02) -- (0,1.00) node[above] {};
\foreach \x/\xtext in {0.25/100, 0.5/200, 0.75/300, 1.0/400}
\draw[shift={(\x,0)}] (0pt,0.2pt) -- (0pt,-0.2pt) node[below] {$\xtext$};
\foreach \y/\ytext in {0.2/9850, 0.4/9900, 0.6/9950, 0.8/10000, 1.0/10050}
\draw[shift={(0,\y)}] (0.2pt,0pt) -- (-0.2pt,0pt) node[left] {$\ytext$};
\draw ( 0.025 , 0.2944 ) node[] ( A042 ) {\textbullet};    \draw ( 0.025 , 0.9152 ) node[] ( B042 ) {$\mathbf{\times}$};
\draw ( 0.05 , 0.1496 ) node[] ( A043 ) {\textbullet};    \draw ( 0.05 , 0.7668 ) node[] ( B043 ) {$\mathbf{\times}$};
\draw ( 0.075 , 0.1300 ) node[] ( A044 ) {\textbullet};    \draw ( 0.075 , 0.6596 ) node[] ( B044 ) {$\mathbf{\times}$};
\draw ( 0.1 , 0.1236 ) node[] ( A045 ) {\textbullet};    \draw ( 0.1 , 0.5760 ) node[] ( B045 ) {$\mathbf{\times}$};
\draw ( 0.125 , 0.1204 ) node[] ( A046 ) {\textbullet};    \draw ( 0.125 , 0.5076 ) node[] ( B046 ) {$\mathbf{\times}$};
\draw ( 0.15 , 0.1184 ) node[] ( A047 ) {\textbullet};    \draw ( 0.15 , 0.4508 ) node[] ( B047 ) {$\mathbf{\times}$};
\draw ( 0.175 , 0.1180 ) node[] ( A048 ) {\textbullet};    \draw ( 0.175 , 0.4012 ) node[] ( B048 ) {$\mathbf{\times}$};
\draw ( 0.2 , 0.1180 ) node[] ( A049 ) {\textbullet};    \draw ( 0.2 , 0.3596 ) node[] ( B049 ) {$\mathbf{\times}$};
\draw ( 0.225 , 0.1176 ) node[] ( A050 ) {\textbullet};    \draw ( 0.225 , 0.3252 ) node[] ( B050 ) {$\mathbf{\times}$};
\draw ( 0.25 , 0.1172 ) node[] ( A051 ) {\textbullet};    \draw ( 0.25 , 0.2944 ) node[] ( B051 ) {$\mathbf{\times}$};
\draw ( 0.275 , 0.1160 ) node[] ( A052 ) {\textbullet};    \draw ( 0.275 , 0.2676 ) node[] ( B052 ) {$\mathbf{\times}$};
\draw ( 0.3 , 0.1160 ) node[] ( A053 ) {\textbullet};    \draw ( 0.3 , 0.2444 ) node[] ( B053 ) {$\mathbf{\times}$};
\draw ( 0.325 , 0.1156 ) node[] ( A054 ) {\textbullet};    \draw ( 0.325 , 0.2268 ) node[] ( B054 ) {$\mathbf{\times}$};
\draw ( 0.35 , 0.1152 ) node[] ( A055 ) {\textbullet};    \draw ( 0.35 , 0.2128 ) node[] ( B055 ) {$\mathbf{\times}$};
\draw ( 0.375 , 0.1152 ) node[] ( A056 ) {\textbullet};    \draw ( 0.375 , 0.2004 ) node[] ( B056 ) {$\mathbf{\times}$};
\draw ( 0.4 , 0.1152 ) node[] ( A057 ) {\textbullet};    \draw ( 0.4 , 0.1912 ) node[] ( B057 ) {$\mathbf{\times}$};
\draw ( 0.425 , 0.1152 ) node[] ( A058 ) {\textbullet};    \draw ( 0.425 , 0.1844 ) node[] ( B058 ) {$\mathbf{\times}$};
\draw ( 0.45 , 0.1152 ) node[] ( A059 ) {\textbullet};    \draw ( 0.45 , 0.1792 ) node[] ( B059 ) {$\mathbf{\times}$};
\draw ( 0.475 , 0.1152 ) node[] ( A060 ) {\textbullet};    \draw ( 0.475 , 0.1740 ) node[] ( B060 ) {$\mathbf{\times}$};
\draw ( 0.5 , 0.1152 ) node[] ( A061 ) {\textbullet};    \draw ( 0.5 , 0.1692 ) node[] ( B061 ) {$\mathbf{\times}$};
\draw ( 0.525 , 0.1152 ) node[] ( A062 ) {\textbullet};    \draw ( 0.525 , 0.1652 ) node[] ( B062 ) {$\mathbf{\times}$};
\draw ( 0.55 , 0.1152 ) node[] ( A063 ) {\textbullet};    \draw ( 0.55 , 0.1616 ) node[] ( B063 ) {$\mathbf{\times}$};
\draw ( 0.575 , 0.1152 ) node[] ( A064 ) {\textbullet};    \draw ( 0.575 , 0.1580 ) node[] ( B064 ) {$\mathbf{\times}$};
\draw ( 0.6 , 0.1152 ) node[] ( A065 ) {\textbullet};    \draw ( 0.6 , 0.1556 ) node[] ( B065 ) {$\mathbf{\times}$};
\draw ( 0.625 , 0.1152 ) node[] ( A066 ) {\textbullet};    \draw ( 0.625 , 0.1532 ) node[] ( B066 ) {$\mathbf{\times}$};
\draw ( 0.65 , 0.1152 ) node[] ( A067 ) {\textbullet};    \draw ( 0.65 , 0.1504 ) node[] ( B067 ) {$\mathbf{\times}$};
\draw ( 0.675 , 0.1152 ) node[] ( A068 ) {\textbullet};    \draw ( 0.675 , 0.1488 ) node[] ( B068 ) {$\mathbf{\times}$};
\draw ( 0.7 , 0.1152 ) node[] ( A069 ) {\textbullet};    \draw ( 0.7 , 0.1468 ) node[] ( B069 ) {$\mathbf{\times}$};
\draw ( 0.725 , 0.1148 ) node[] ( A070 ) {\textbullet};    \draw ( 0.725 , 0.1460 ) node[] ( B070 ) {$\mathbf{\times}$};
\draw ( 0.75 , 0.1148 ) node[] ( A071 ) {\textbullet};    \draw ( 0.75 , 0.1436 ) node[] ( B071 ) {$\mathbf{\times}$};
\draw ( 0.775 , 0.1148 ) node[] ( A072 ) {\textbullet};    \draw ( 0.775 , 0.1424 ) node[] ( B072 ) {$\mathbf{\times}$};
\draw ( 0.8 , 0.1148 ) node[] ( A073 ) {\textbullet};    \draw ( 0.8 , 0.1404 ) node[] ( B073 ) {$\mathbf{\times}$};
\draw ( 0.825 , 0.1148 ) node[] ( A074 ) {\textbullet};    \draw ( 0.825 , 0.1400 ) node[] ( B074 ) {$\mathbf{\times}$};
\draw ( 0.85 , 0.1148 ) node[] ( A075 ) {\textbullet};    \draw ( 0.85 , 0.1396 ) node[] ( B075 ) {$\mathbf{\times}$};
\draw ( 0.875 , 0.1148 ) node[] ( A076 ) {\textbullet};    \draw ( 0.875 , 0.1380 ) node[] ( B076 ) {$\mathbf{\times}$};
\draw ( 0.9 , 0.1148 ) node[] ( A077 ) {\textbullet};    \draw ( 0.9 , 0.1364 ) node[] ( B077 ) {$\mathbf{\times}$};
\draw ( 0.925 , 0.1148 ) node[] ( A078 ) {\textbullet};    \draw ( 0.925 , 0.1360 ) node[] ( B078 ) {$\mathbf{\times}$};
\draw ( 0.95 , 0.1148 ) node[] ( A079 ) {\textbullet};    \draw ( 0.95 , 0.1356 ) node[] ( B079 ) {$\mathbf{\times}$};
\draw ( 0.975 , 0.1148 ) node[] ( A080 ) {\textbullet};    \draw ( 0.975 , 0.1344 ) node[] ( B080 ) {$\mathbf{\times}$};
\draw ( 1 , 0.1148 ) node[] ( A081 ) {\textbullet};    \draw ( 1 , 0.1336 ) node[] ( B081 ) {$\mathbf{\times}$};

\end{tikzpicture} 
\caption{SS type} 
\label{fig07:05}
\end{center}
\end{subfigure}
\caption{Comparison of upper bounds for \textbf{(SOCKP)} between our approach (\textbullet) and Han et al. \cite{Han16} ($\mathbf{\times}$), respectively. Each point represents the upper bound (UB) for $m=10,20,\cdots,400$, when $n=400$, $\rho=0.95$.}
\label{fig07_01}
\end{figure}

Figure \ref{fig07_01} shows a comparison of the upper bounds obtained by our approach and Han et al.'s \cite{Han16} when $n=400$, $\rho=0.95$. Our algorithm spent at most 10\% more time than Han et al.'s  \cite{Han16}. However, our approach outperformed Han et al.'s \cite{Han16} in terms of the quality of the upper bounds. The upper bounds obtained when $m=10$ in our approach were almost the same as the upper bounds obtained when $m=100$ in their method. The upper bounds also obtained when $m=20$ in our approach  were almost the same as the upper bounds obtained when $m=150$--$250$ in their method.

\setlength{\tabcolsep}{5pt}
\renewcommand{\arraystretch}{1.0}
\begin{table}[]
\caption{Computational results of Algorithm \ref{algo06_01} for \textbf{(SOCKP)}.}
\label{table07:06}
\center
\begin{tabular}{ccc@{\extracolsep{10pt}}rrrrr}
\hline 
\multirow{2}{*}{$n$} & \multirow{2}{*}{type} & \multirow{2}{*}{$\rho$} & \multicolumn{5}{c}{exact}   \\
 & & & \multicolumn{1}{c}{\#itr} & \multicolumn{1}{c}{$\hat{m}$} & \multicolumn{1}{c}{\#knap} & \multicolumn{1}{c}{time(s)} & \multicolumn{1}{c}{obj} \\ \hline
\multirow{10}{*}{100} & \multirow{2}{*}{SC} & 0.95 & 2.9 & 44 & 7,144 & 0.42 & 3,093.6 \\
 & & 0.99 & 3.6 & 74 & 12,628 & 0.69 & 2,949.9 \\ \cline{2-3} \cline{4-8}
 & \multirow{2}{*}{SCR} & 0.95 & 3.4 & 56 & 9,256 & 0.53 & 3,011.9 \\
 & & 0.99 & 3.9 & 84 & 14,338 & 0.79 & 2,760.2 \\ \cline{2-3} \cline{4-8}
 & \multirow{2}{*}{IC} & 0.95 & 2.7 & 36 & 5,655 & 0.29 & 2,619.6 \\
 & & 0.99 & 3.2 & 64& 10,756 & 0.58 & 2,450.7 \\ \cline{2-3} \cline{4-8}
 & \multirow{2}{*}{ICR}  & 0.95 & 2.3 & 33 & 5,095 & 0.32 & 2,720.1 \\
 & & 0.99 & 2.5 & 37 & 5,821 & 0.33 & 2,667.2 \\ \cline{2-3} \cline{4-8}
 & \multirow{2}{*}{SS} & 0.95 & 3.6 & 70 & 11,890 & 0.67 & 2,297.7 \\
 & & 0.99 & 3.2 & 48 & 7,870 & 0.42 & 2,135.9 \\ \hline
\multirow{10}{*}{400} & \multirow{2}{*}{SC} & 0.95 & 3.8 & 152 & 108,801 & 11.45 & 12,750.5 \\
 & & 0.99 & 4.3 & 248 & 182,274 & 28.07 & 12,440.5 \\ \cline{2-3} \cline{4-8}
 & \multirow{2}{*}{SCR} & 0.95 & 3.7 & 136 & 96,139 & 9.96 & 12,600.7 \\
 & & 0.99 & 4.4 & 240 & 175,480 & 23.66 & 12,086.3 \\ \cline{2-3} \cline{4-8}
 & \multirow{2}{*}{IC} & 0.95 & 3.5 & 132 & 93,286 & 9.29 & 10,805.9 \\
 & & 0.99 & 4.0 & 176 & 126,915 & 13.15 & 10,437.9 \\ \cline{2-3} \cline{4-8}
 & \multirow{2}{*}{ICR} & 0.95 & 2.4 & 56 & 35,146 & 3.37 & 11,014.6 \\
 & & 0.99 & 2.9 & 84 & 56,526 & 4.92 & 10,900.8 \\ \cline{2-3} \cline{4-8}
 & \multirow{2}{*}{SS} & 0.95 & 3.9 & 160 & 114,810 & 14.02 & 9,828.7 \\
 & & 0.99 & 4.4 & 240 & 176,025 & 27.86 & 9,468.3 \\ \hline
\multirow{10}{*}{900} & \multirow{2}{*}{SC} & 0.95 & 4.0 & 300 & 498,486 & 185.72 & 28,735.0 \\
 & & 0.99 & 4.8 & 480 & 813,260 & 355.42 & 28,260.3 \\ \cline{2-3} \cline{4-8}
 & \multirow{2}{*}{SCR} & 0.95 & 4.7 & 480 & 811,104 & 411.38 & 28,505.9 \\
 & & 0.99 & 4.6 & 408 & 685,524 & 255.96 & 27,726.7 \\ \cline{2-3} \cline{4-8}
 & \multirow{2}{*}{IC} & 0.95 & 3.3 & 168 & 267,323 & 55.17 & 24,083.5 \\
 & & 0.99 & 4.1 & 336 & 560,748 & 244.62 & 23,518.6 \\ \cline{2-3} \cline{4-8}
 & \multirow{2}{*}{ICR} & 0.95 & 2.8 & 273 & 450,275 & 467.14 & 24,395.7 \\
 & & 0.99 & 3.7 & 228 & 371,784 & 110.48 & 24,222.2 \\ \cline{2-3} \cline{4-8}
 & \multirow{2}{*}{SS} & 0.95 & 4.2 & 300 & 498,220 & 180.88 & 22,302.7 \\
 & & 0.99 & 4.3 & 312 & 519,204 & 175.91 & 21,749.3 \\ \hline 
\end{tabular}
\end{table}
Table \ref{table07:06} shows the computational results of Algorithm \ref{algo06_01} when $n = 100, 400$, and $900$. The number of iterations (\#itr) and the last value of $\hat{m}$ are obtained after Algorithm \ref{algo06_01} terminates, and we take the average value of 10 instances for each type and size. We also report the average number of the ordinary BKPs (\#knap) solved, the average computational time of the algorithm in seconds (time(s)), and the average optimal value of \textbf{(SOCKP)} (obj). The average number of iterations was at most five. The exact algorithm obtains an optimal solution within a second, 30 seconds, and 600 seconds when $n = 100, 400$, and $900$, respectively. These results show that Algorithm \ref{algo06_01} works well to solve the SOC-constrained BKP without using the theoretical value of $m^{*}$ to guarantee optimality.

Table \ref{table07:04} illustrates the computational results of CPLEX solving a mixed-integer quadratic programming reformulation of \textbf{(SOCKP)}, when $n = 100, 400$, and $900$. We report the number of nodes in the branch-and-bound tree (\#node), the computational time of the algorithm in seconds (time(s)), the upper and lower bounds ($\textrm{UB}$ and $\textrm{LB}$), and the percentage gap between $\textrm{UB}$ and $\textrm{LB}$ (gap(\%)). We set the time limit to 3,600 seconds, and the number of unsolved instances within the time limit is recorded in the parentheses. The unsolved instances are excluded from the calculation of the average computational time, but they are included in the calculations of the average number of nodes and the average value of bounds.

\setlength{\tabcolsep}{5pt}
\renewcommand{\arraystretch}{1.1}
\begin{table}[]
\caption{Computational results of CPLEX for \textbf{(SOCKP)}.}
\label{table07:04}
\center
\begin{tabular}{ccc@{\extracolsep{10pt}}rrrrr}
\hline 
\multirow{2}{*}{$n$} & \multirow{2}{*}{type} & \multirow{2}{*}{$\rho$} & \multicolumn{1}{c}{\multirow{2}{*}{\#node}}  & \multicolumn{1}{c}{\multirow{2}{*}{time(s)}} & \multicolumn{1}{c}{\multirow{2}{*}{UB}} & \multicolumn{1}{c}{\multirow{2}{*}{LB}}  & \multicolumn{1}{c}{\multirow{2}{*}{gap(\%)}}  \\
 & & & &  & & &  \\  \hline
\multirow{10}{*}{100} & \multirow{2}{*}{SC} & 0.95 & 407,964 & 21.77\textcolor{white}{(0)} & 3,093.6 & 3,093.6 & 0.00 \\
 & & 0.99 & 2,128,899 & 145.44\textcolor{white}{(0)} & 2,949.9 & 2,949.9 & 0.00 \\ \cline{2-3} \cline{4-8}
 & \multirow{2}{*}{SCR} & 0.95 & 684,796 & 22.28\textcolor{white}{(0)} & 3,011.9 & 3,011.9 & 0.00 \\
 & & 0.99 & 22,347,064 & 412.67(4) & 2,763.0 & 2,760.0 & 0.11 \\ \cline{2-3} \cline{4-8}
 & \multirow{2}{*}{IC} & 0.95 & 128,576 & 5.35\textcolor{white}{(0)} & 2,619.6 & 2,619.6 & 0.00 \\
 & & 0.99 & 7,855,418 & 418.14\textcolor{white}{(0)} & 2,450.7 & 2,450.7 & 0.00 \\ \cline{2-3} \cline{4-8}
 & \multirow{2}{*}{ICR} & 0.95 & 488,887 & 17.92\textcolor{white}{(0)} & 2,720.1 & 2,720.1 & 0.00 \\
 & & 0.99 & 35,636 & 1.78\textcolor{white}{(0)} & 2,667.2 & 2,667.2 & 0.00 \\ \cline{2-3} \cline{4-8}
 & \multirow{2}{*}{SS} & 0.95 & 24,794,867 & -(10) & 2,309.2 & 2,297.2 & 0.52 \\
 & & 0.99 & 20,280,971 & 1,362.92(9) & 2,162.1 & 2,135.0 & 1.27 \\ \hline
\multirow{10}{*}{400} & \multirow{2}{*}{SC}  & 0.95 & 12,329,385 & -(10) & 12,754.7 & 12,750.4 & 0.03 \\
 & & 0.99 & 8,898,703 & -(10) & 12,466.8 & 12,438.9 & 0.22 \\ \cline{2-3} \cline{4-8}
 & \multirow{2}{*}{SCR} & 0.95 & 18,090,376 & 127.23(7) & 12,601.6 & 12,600.6 & 0.01 \\
 & & 0.99 & 16,311,961 & 458.24(9) & 12,087.7 & 12,086.2 & 0.01 \\ \cline{2-3} \cline{4-8}
 & \multirow{2}{*}{IC} & 0.95 & 18,671,523 & -(10) & 10,808.6 & 10,805.5 & 0.03 \\
 & & 0.99 & 14,708,954 & -(10) & 10,460.0 & 10,436.6 & 0.22 \\ \cline{2-3} \cline{4-8}
 & \multirow{2}{*}{ICR} & 0.95 & 16,732,122 & 756.79(7) & 11,016.0 & 11,014.3 & 0.02 \\
 & & 0.99 & 14,452,955 & 1,906.14(4) & 10,901.4 & 10,900.8 & 0.01 \\ \cline{2-3} \cline{4-8}
 & \multirow{2}{*}{SS} & 0.95 & 6,854,149 & -(10) & 9,891.3 & 9,827.7 & 0.65 \\
 & & 0.99 & 8,956,697 & -(10) & 9,605.8 & 9,466.8 & 1.47 \\ \hline
\multirow{10}{*}{900} & \multirow{2}{*}{SC} & 0.95 & 5,115,085 & -(10) & 28,741.0 & 28,734.5 & 0.02 \\
 & & 0.99 & 4,550,226 & -(10) & 28,293.1 & 28,259.0 & 0.12 \\ \cline{2-3} \cline{4-8}
 & \multirow{2}{*}{SCR} & 0.95 & 13,136,797 & 2,228.71(9) & 28,508.2 & 28,505.8 & 0.01 \\
 & & 0.99 & 11,900,707 & -(10) & 27,729.4 & 27,726.6 & 0.01 \\ \cline{2-3} \cline{4-8}
 & \multirow{2}{*}{IC} & 0.95 & 13,221,524 & -(10) & 24,086.7 & 24,083.1 & 0.02 \\
 & & 0.99 & 8,762,550 & -(10) & 23,549.8 & 23,517.6 & 0.14 \\ \cline{2-3} \cline{4-8}
 & \multirow{2}{*}{ICR} & 0.95 & 13,919,960 & -(10) & 24,397.7 & 24,395.2 & 0.01 \\
 & & 0.99 & 7,369,902 & -(10) & 24,227.1 & 24,222.0 & 0.02 \\ \cline{2-3} \cline{4-8}
 & \multirow{2}{*}{SS} & 0.95 & 2,959,289 & -(10) & 22,403.9 & 22,301.0 & 0.46 \\
 & & 0.99 & 4,090,540 & -(10) & 21,972.5 & 21,746.0 & 1.04 \\ \hline
\end{tabular}\\ 
\flushleft
($\cdot$): number of instances not solved within 3,600 seconds out of ten 
\end{table}

We can compare the computational results of Algorithm \ref{algo06_01} shown in Table \ref{table07:06} and those of CPLEX shown in Table \ref{table07:04}. When $n=100$, CPLEX obtained exact optimal solutions within the time limit for the instances except for some instances of SCR and SS types. However, CPLEX took much more time than our method. It seems that SS type was particularly difficult to solve using CPLEX. However, our exact algorithm did not show any significant difference in performance depending on instance types. 

When $n=400$ and $n=900$, most of the instances were not solved exactly within 3,600 seconds by CPLEX. The gaps were less than 0.2\% for the instances except for the SS type, but our exact algorithm could solve these instances within 30 seconds when $n=400$ and 8 minutes when $n=900$, respectively. In particular, for the instances of SC, IC and SS types, CPLEX was not able to obtain an optimal solution for all instances within the time limit, and sometimes the gaps were more than 1\%, whereas our exact algorithm could solve all of them within an hour. Moreover, we observed that when CPLEX could not obtain an optimal solution, the final gap value was no better than the gap value obtained by our approximation algorithm within a minute as shown in Table \ref{table07:01}.

\setlength{\tabcolsep}{5pt}
\renewcommand{\arraystretch}{1.1}
\begin{table}[]
\caption{Computational results of solving \textbf{(RKPm)} to obtain the upper and lower bounds for large-sized \textbf{(SOCKP)} simultaneously for different $m$ values.}
\label{table07:05}
\center
\begin{tabular}{ccc@{\extracolsep{4pt}}rr@{\extracolsep{4pt}}rr@{\extracolsep{4pt}}rr}
\hline 
\multirow{2}{*}{$n$} & \multirow{2}{*}{type} & \multirow{2}{*}{$\rho$} & \multicolumn{2}{c}{$m=\sqrt{n}/2$} & \multicolumn{2}{c}{$m=\sqrt{n}$} & \multicolumn{2}{c}{$m=2\sqrt{n}$} \\ \cline{4-5}  \cline{6-7} \cline{8-9}
 & & & \multicolumn{1}{c}{time(s)} & \multicolumn{1}{c}{gap(\%)} & \multicolumn{1}{c}{time(s)} & \multicolumn{1}{c}{gap(\%)} & \multicolumn{1}{c}{time(s)} & \multicolumn{1}{c}{gap(\%)} \\ \hline
\multirow{10}{*}{2,500} & \multirow{2}{*}{SC} & 0.95  & 16.98 & 0.38  & 40.05 & 0.10  & 109.17 & 0.03 \\
 & & 0.99 & 15.15 & 0.86 & 36.91 & 0.22 & 101.93 & 0.06 \\ \cline{2-3} \cline{4-5}  \cline{6-7} \cline{8-9}
 & \multirow{2}{*}{SCR} & 0.95 & 20.72 & 0.73 & 46.75 & 0.15 & 122.05 & 0.04 \\
 & & 0.99 & 16.57 & 1.70 & 37.65 & 0.36 & 103.85 & 0.09 \\ \cline{2-3} \cline{4-5}  \cline{6-7} \cline{8-9}
 & \multirow{2}{*}{IC} & 0.95 & 15.27 & 0.50 & 37.48 & 0.13 & 102.45 & 0.03 \\
 & & 0.99 & 14.78 & 1.14  & 36.32 & 0.30 & 100.42 & 0.08 \\ \cline{2-3} \cline{4-5}  \cline{6-7} \cline{8-9}
 & \multirow{2}{*}{ICR} & 0.95 & 22.20 & 0.15 & 50.32 & 0.04 & 126.93 & 0.01 \\
 & & 0.99 & 17.35 & 0.33 & 39.46 & 0.09 & 107.01 & 0.02 \\ \cline{2-3} \cline{4-5}  \cline{6-7} \cline{8-9}
 & \multirow{2}{*}{SS} & 0.95 & 39.05 & 0.58 & 83.23 & 0.15 & 191.41 & 0.04 \\
 & & 0.99 & 41.14 & 1.31 & 89.23 & 0.33 & 203.87 & 0.08 \\ \hline 
\multirow{10}{*}{4,900} & \multirow{2}{*}{SC} & 0.95 & 89.77 & 0.27 & 244.82 & 0.07 & 749.41 & 0.02 \\
 & & 0.99 & 82.47 & 0.62 & 230.42 & 0.16 & 721.10 & 0.04 \\ \cline{2-3} \cline{4-5}  \cline{6-7} \cline{8-9}
 & \multirow{2}{*}{SCR} & 0.95 & 105.80 & 0.51  & 273.04 & 0.11 & 802.09 & 0.03 \\
 & & 0.99 & 88.58 & 1.19 & 232.83 & 0.25  & 726.59 & 0.06 \\ \cline{2-3} \cline{4-5}  \cline{6-7} \cline{8-9}
 & \multirow{2}{*}{IC} & 0.95 & 81.93 & 0.36 & 232.84 & 0.09 & 722.73 & 0.02 \\
 & & 0.99 & 80.39 & 0.81 & 228.78 & 0.22 & 718.19 & 0.06 \\ \cline{2-3} \cline{4-5}  \cline{6-7} \cline{8-9}
 & \multirow{2}{*}{ICR} & 0.95 & 111.18 & 0.11 & 281.06 & 0.03 & 834.72 & 0.01 \\
 & & 0.99 & 88.30 & 0.24 & 237.70 & 0.07 & 763.79 & 0.02 \\ \cline{2-3} \cline{4-5}  \cline{6-7} \cline{8-9}
 & \multirow{2}{*}{SS} & 0.95 & 217.25 & 0.41 & 507.58 & 0.11 & 1,244.87 & 0.03 \\
 & & 0.99 & 228.54 & 0.94 & 522.26 & 0.24 & 1,294.14 & 0.06 \\ \hline 
\multirow{10}{*}{10,000} & \multirow{2}{*}{SC} & 0.95 & 569.94 & 0.19 & 1,723.30 & 0.05 & & \\
 & & 0.99 & 542.39 & 0.43 & 1,661.84 & 0.11 & & \\ \cline{2-3} \cline{4-5}  \cline{6-7} 
 & \multirow{2}{*}{SCR} & 0.95 & 666.90 & 0.36 & 2,000.91 & 0.08 & & \\
 & & 0.99 & 579.38 & 0.83 & 1,809.82 & 0.17 & & \\ \cline{2-3} \cline{4-5}  \cline{6-7} 
 & \multirow{2}{*}{IC} & 0.95 & 554.75 & 0.25 & 1,684.05 & 0.07 & & \\
 & & 0.99 & 545.38 & 0.57 & 1,691.10 & 0.15 & & \\ \cline{2-3} \cline{4-5}  \cline{6-7} 
 & \multirow{2}{*}{ICR} & 0.95 & 650.89 & 0.08 & 1,890.47 & 0.02 & & \\
 & & 0.99 & 571.86 & 0.17 & 1,683.17 & 0.05 & & \\ \cline{2-3} \cline{4-5}  \cline{6-7} 
 & \multirow{2}{*}{SS} & 0.95 & 1,344.24 & 0.29 & 3,435.69 & 0.07 & & \\
 & & 0.99 & 1,423.77 & 0.66 & 3,495.98 & 0.17 & & \\ \cline{1-7}
\end{tabular}
\end{table}

We also tested our algorithm to obtain the upper and lower bounds for \textbf{(SOCKP)} simultaneously for large problems. Table \ref{table07:05} shows the computational results for $m=\sqrt{n}/2, \sqrt{n}, 2\sqrt{n}$ when $n = 2,500$ and $n=4,900$, and $m=\sqrt{n}/2, \sqrt{n}$ when $n = 10,000$. We can observe that the gap was less than $0.4\%$, $0.25\%$, and $0.2\%$ in the case of $m=\sqrt{n}$ when $n=2,500$, $n=4,900$, and $n=10,000$, respectively. The required computational time was less than 90 seconds, 600 seconds, and 3,600 seconds when $n=2,500$, $n=4,900$, and $n=10,000$, respectively. Moreover, In the case of $m=2\sqrt{n}$, the gap became less than $0.1\%$ when $n=2,500$ and $n=4,900$. 

\section{Conclusion} \label{sect08}

In this paper, we considered the chance-constrained binary knapsack problem where the weights of items are random and independently distributed, and their means and standard deviations are given. The problem can be reformulated as the second-order cone-constrained binary knapsack problem \textbf{(SOCKP)} when the distributions of weights follow some specified assumptions in stochastic programming or distributionally robust optimization. The problem can be reformulated as the robust binary knapsack problem with the ellipsoidal uncertainty set, and the upper and lower bounds for \textbf{(SOCKP)} can be obtained from the robust binary knapsack problems \textbf{(RKPm)} with inner and outer approximations of the ellipsoidal uncertainty set, respectively. \textbf{(RKPm)} can be solved by solving $\mathrm{O}(mn)$ ordinary binary knapsack problems. In addition, we demonstrated some probability guarantees of feasible solutions to \textbf{(RKPm)} theoretically. We also demonstrated the existence of a pseudo-polynomial time algorithm for \textbf{(SOCKP)} and proposed an exact algorithm by increasing the value of $m$ and solving \textbf{(RKPm)} iteratively until an optimal solution to \textbf{(RKPm)} with $\mathrm{\Delta} = m^{2}$ also becomes optimal to \textbf{(SOCKP)}. The computational results show that very small gap values can be obtained within a short time. The exact algorithm for \textbf{(SOCKP)} significantly outperforms CPLEX. 

For future research, improvements of our exact algorithm may be needed. Moreover, it would be interesting to consider whether our approach can be extended to solve second-order cone-constrained binary knapsack problems with a non-diagonal covariance matrix (i.e. random variables for the chance-constrained problem are dependent). Applications of our approach to solve other types of stochastic and distributionally robust problems not considered in this paper may be worthwhile too.





\end{document}